\documentclass{amsart}
\usepackage{amscd,amssymb,stmaryrd}
\usepackage[all]{xy}
\renewcommand{\mod}{\operatorname{mod}\nolimits}
\newcommand{\Mod}{\operatorname{Mod}\nolimits}

\newcommand{\gr}{{\operatorname{gr}\nolimits}}
\newcommand{\Gr}{{\operatorname{Gr}\nolimits}}

\newcommand{\lfGr}{{\operatorname{lfGr}\nolimits}}
\newcommand{\add}{\operatorname{add}\nolimits}

\newcommand{\Hom}{\operatorname{Hom}\nolimits}
\newcommand{\End}{\operatorname{End}\nolimits}

\renewcommand{\Im}{\operatorname{Im}\nolimits}
\newcommand{\Ker}{\operatorname{Ker}\nolimits}

\newcommand{\rrad}{\mathfrak{r}}
\newcommand{\rad}{\operatorname{rad}\nolimits}

\newcommand{\Ext}{\operatorname{Ext}\nolimits}
\newcommand{\Tor}{\operatorname{Tor}\nolimits}

\newcommand{\op}{{\operatorname{op}\nolimits}}
\newcommand{\Ab}{{\operatorname{Ab}\nolimits}}

\newcommand{\comp}{\operatorname{\scriptstyle\circ}}

\newcommand{\G}{\Gamma}
\renewcommand{\L}{\Lambda}
\newcommand{\Z}{{\mathbb Z}}
\newcommand{\A}{{\mathcal A}}
\newcommand{\C}{{\mathcal C}}
\newcommand{\D}{{\mathcal D}}
\newcommand{\E}{{\mathcal E}}
\newcommand{\I}{{\mathcal I}}

\newcommand{\Li}{{\mathcal L}}
\newcommand{\wK}{w{\mathcal K}}

\renewcommand{\S}{{\mathcal S}}
\newcommand{\T}{{\mathcal T}}

\newcommand{\extto}{\xrightarrow}

\newtheorem{lem}{Lemma}[section]
\newtheorem{prop}[lem]{Proposition}
\newtheorem{cor}[lem]{Corollary}
\newtheorem{thm}[lem]{Theorem}
\theoremstyle{definition}
\newtheorem{defin}[lem]{Definition}

\newtheorem{example}[lem]{Example}
\newtheorem*{remark}{Remark}

\begin{document}

\title{Graded and Koszul categories}
\author[Mart\'inez-Villa]{Roberto Mart\'inez-Villa}
\address{Roberto Mart\'inez-Villa, Instituto de Matem\'aticas,
  Universidad Nacional Autonoma de Mexico, Campus Morelia,
Apartado Postal 27-3 (Xangari), C.P. 58089, Morelia, Michoac\'an, Mexico}
\email{mvilla@matmor.unam.mx}
\thanks{The first author thanks the Universidad Nacional Autonoma de
  Mexico program PAPITT for funding
  the research project. In addition he thanks his coauthor and
  Department of Mathematical Sciences (NTNU) for their kind
  hospitality and support through Research Council
  of Norway Storforsk grant no.\ 167130.}
\author[Solberg]{\O yvind Solberg}
\address{\O yvind Solberg\\Institutt for matematiske fag\\
NTNU\\ N--7034 Trondheim\\ Norway}
\email{oyvinso@math.ntnu.no}
\thanks{The second author thanks Universidad Nacional Autonoma de
Mexico and his coauthor for their kind hospitality and support, and
the Department of Mathematical Sciences (NTNU) and Research Council
of Norway Storforsk grant no.\ 167130 also for the support.}
\date{\today}

\begin{abstract}
Koszul algebras have arisen in many contexts; algebraic geometry,
combinatorics, Lie algebras, non-commutative geometry and
topology. The aim of this paper and several sequel papers is to show
that for any finite dimensional algebra there is always a naturally
associated Koszul theory. To obtain this, the notions of Koszul
algebras, linear modules and Koszul duality are extended to additive
(graded) categories over a field. The main focus of this paper is to
provide these generalizations and the necessary preliminaries.
\end{abstract}

\maketitle

\section*{Introduction}
Koszul theory is usually applied to graded algebras $\L$ which are
Koszul. The theory has been extended to non-graded semiperfect
Noetherian algebras $\L$ through the notion of weakly Koszul algebras
\cite{MVZ}.  The aim of this paper is to show that Koszul theory can
be applied to any finite dimensional algebra by associating a Koszul
object and therefore a Koszul theory for any finite dimensional
algebra. This theory is found by considering the category of all
additive contravariant functors from finitely generated $\L$-modules
to vector spaces. The simple objects in this category are known to be
weakly Koszul by a result of Igusa-Todorov (\cite{IT}). Similarly as
for algebras, we then pass to a naturally associated graded category,
where the simple objects are linear. In this way Koszul theory can be
applied to study any finite dimensional algebra. This application
serves as a motivation for most of the definitions and the results in
the present paper and the subsequent papers based on the current
paper. The generalization of the Koszul theory we introduce goes
through extending the notions of Koszul algebras, linear modules and
Koszul duality, to additive $K$-categories over a field $K$. For
related work, but a different focus, we point out the work of
Mazorchuk, Ovsienko and Stroppel in \cite{MOS}.

The process of associating a Koszul object to any finite dimensional
algebra goes through utilizing the analogy with algebras. As mentioned
above, Koszul theory has been extended to non-graded finite
dimensional algebras $\L$ through the notion of weakly Koszul algebras
\cite{MVZ}: $\L$ is \emph{weakly Koszul} if all simple $\L$-modules
$S$ have a minimal projective resolution $\cdots \to P_2\to P_1\to
P_0\to S\to 0$ satisfying $\rrad^{i+1}P_j\cap
\Omega^{j+1}_\L(S)=\rrad^i\Omega^{j+1}_\L(S)$ for all $j\geq 0$ and
$i\geq 0$, where $\rrad$ is the Jacobson radical of $\L$. If $\L$ is
weakly Koszul, then the associated graded ring
$\A_\gr(\L)=\amalg_{i\geq 0} \rrad^i/\rrad^{i+1}$ is Koszul. Extending
this to $\Mod(\mod\L)$, the category of all additive contravariant
functors from the category of finitely generated $\L$-modules,
$\mod\L$, to vector spaces over the field and using results of
Igusa-Todorov \cite{IT}, we show that $\Mod(\mod\L)$ is weakly
Koszul. Again as for algebras, we show that the naturally associated
graded category is Koszul. To demystify this object we describe it for
finite representation type. For a finite dimensional algebra $\L$ of
finite representation type, the category $\Mod(\mod\L)$ is equivalent
to the category of modules over the Auslander algebra $\G$ of
$\L$. Which by the results of Igusa-Todorov is weakly Koszul. Then the
associated Koszul object we describe for this finite dimensional
algebra is equivalent to the module category of graded modules over
the associated graded ring $\A_\gr(\G)$, which is Koszul.

Next we describe the organization of the paper. In Section
\ref{section1} we recall definitions of graded categories and
functors, and in addition discuss fundamental concepts and results as
Yoneda's Lemma, ideals, tensor products, Nakayama's Lemma. While
Section \ref{section1.5} deals with projective and simple objects,
duality and homological dimensions. Section \ref{section2} is devoted
to defining and proving basic results about Koszul categories, where
we end with a brief discussion on our main application appearing in a
subsequent paper.  An analogue of weakly Koszul algebras for
categories is introduced in Section \ref{section3}. The graded
categories we consider are generated in degrees $0$ and $1$, and with
the further assumption we impose they are quotients of free tensor
categories over a bimodule. These categories and the Koszul dual of
such are discussed in the the last section.

Finally in this introduction we mention the standing assumptions
throughout the paper. An additive (graded) $K$-category $\C$ is said
to be \emph{Krull-Schmidt} if any object in $\C$ is a finite direct
sum of objects with a (graded) local endomorphism ring. Throughout we
are assuming that in any category we consider, all idempotents
split. Under this assumption using \cite[Theorem 27.6]{AF} it follows
that an additive (graded) $K$-category is Krull-Schmidt if and only if
$\End_\C(C)$ (or in the graded case $\End_\C(C)_0$) is semiperfect for
all $C$ in $\C$. Throughout we assume that all categories we consider
are skeletally small.

\section{Graded categories and functors}\label{section1}

This section is devoted to recalling definitions of graded categories
and functors between graded categories. Throughout the paper $K$
denotes a fixed field. In further detail, after graded categories and
functors between them are introduced, we discuss Yoneda's Lemma,
ideals, tensor products of functors and Nakayama's Lemma. All modules
throughout the paper are left modules unless otherwise explicitly
said otherwise. 

\subsection{Graded categories} First we introduce graded categories.
Let $\C$ be a $K$-category. The category $\C$ is called \emph{graded}
if for each pair of objects $C$ and $D$ in $\C$ we have
\[\Hom_\C(C,D)=\amalg_{i\in\Z}\Hom_\C(C,D)_i\]
as a $\Z$-graded vector space over $K$, such that if $f$ is in
$\Hom_\C(C,C')_i$ and $g$ is in $\Hom_\C(C',C'')_j$, then $gf$ is in
$\Hom_\C(C,C'')_{i+j}$. In particular the identity maps are
concentrated in one degree and this degree is $0$.

Just having a graded $K$-category is normally too general, so further
conditions are often needed. The following two restrictions are
central in what follows.
\begin{defin}
\begin{enumerate}
\item[(a)] A graded $K$-category $\C$ is \emph{locally finite} if
$\Hom_\C(C,D)_j$ is finite dimensional over $K$ for all objects $C$
and $D$ in $\C$ and all integers $j$.
\item[(b)] A graded $K$-category $\C$ is called \emph{positively
graded} if
\[\Hom_\C(X,Y)=\amalg_{i\geq 0}\Hom_\C(X,Y)_i\]
for all objects $X$ and $Y$ in $\C$.
\end{enumerate}
\end{defin}
Below we give some examples of graded categories where we return to some
of them later. To this end we fix the following notation. For a
$\Z$-graded module over a $\Z$-graded ring we define the $i$-th shift
functor $[i]$ as follows: $M[i]_j=M_{i+j}$ and for a homomorphism
$f\colon M\to N$ of graded modules $f[i]=f\colon M[i]\to N[i]$.

\begin{example} The category of graded modules over a graded ring is the
first example we review.  Let $\L=\amalg_{i\geq 0} \L_i$ be a
positively graded algebra over $K$. Denote by $\Gr(\L)$ the category
having as objects all $\Z$-graded $\L$-modules and
\[\Hom_{\Gr(\L)}(M,N) = \amalg_{i\in\Z}\Hom_{\Gr(\L)_0}(M,N[i]),\]
where $\Gr(\L)_0$ denotes the category of graded $\L$-modules with
degree zero homomorphisms. Then $\Gr(\L)$ is a graded category with
\[\Hom_{\Gr(\L)}(M,N)_i=\Hom_{\Gr(\L)_0}(M,N[i])\]
for all $M$ and $N$ in $\Gr(\L)$. Note here that even though $\L$ is
positively graded, the category $\Gr(\L)$ is never positively
graded (as long as $\L\neq (0)$).
\end{example}

\begin{example}\label{exam:assgraded} Here we define the associated
graded category of an additive category $\C$ with respect to the
radical of $\C$. This construction can in fact be done with respect to
any ideal in the category $\C$. See subsection \ref{subsec:ideals} for
a short discussion about ideals in a category.

Let $\C$ be an additive $K$-category, and denote by
$\rad_\C$ the radical of $\C$. Recall that the radical, $\rad_\C(-,-)\colon
\C^\op\times \C\to \Ab$, as a subfunctor of $\Hom_\C(-,-)$ is given by
\[\rad_\C(C,D) = \{f\in\Hom_\C(C,D)\mid gf \in \rad(\End_\C(C)) \text{\ for
all\ }  g \in \Hom_\C(D,C)\}.\]
Observe that we also have
\[\rad_\C(C,D)= \{f\in\Hom_\C(C,D)\mid fh \in
\rad(\End_\C(D)) \text{\ for all\ }  h \in \Hom_\C(D,C)\}.\]
(see \cite{M}). Therefore $\rad_\C=\rad_{\C^\op}$. Furthermore, a
morphism $f$ in $\Hom_\C(C,D)$ is in $\rad_\C^2(C,D)$ if and only if $f$
is a finite sum of maps of the form $C\extto{f'_i} X\extto{f''_i} D$
with $f'_i$ in $\rad_\C(C,X)$ and $f''_i$ in $\rad_\C(X,D)$ for
$i=1,2,\ldots,n$ and $f=\sum_{i=1}^nf''_if'_i$. Inductively define
$\rad_\C^n=\rad_\C\cdot \rad_\C^{n-1}$.

The associated graded category, $\A_\gr(\C)$, of $\C$ (with respect to
the radical) has the same objects as $\C$ while the morphisms are
given by
\[\Hom_{\A_\gr(\C)}(A,B)=\amalg_{i\geq 0}\rad_\C^i(A,B)/\rad_\C^{i+1}(A,B).\]
Then $\A_\gr(\C)$ is a positively graded category with
\[\Hom_{\A_\gr(\C)}(A,B)_i=\rad_\C^i(A,B)/\rad_\C^{i+1}(A,B)\]
for all objects $A$ and $B$ in $\C$.

The prime example and application of this construction for us, is to
consider a finite dimensional algebra $\L$ and let $\C=\Mod(\mod\L)$,
the category of all additive functors from $(\mod\L)^\op$ to $\Mod
K$. In Section \ref{section2} and in subsequent papers we return to
this example. More generally, for a $K$-category $\D$, we later
consider the category, and subcategories of, $\Mod(\D)$, which denotes
the category of all additive functors from $\D^\op$ to $\Mod K$.
\end{example}

\begin{example}\label{exam:extcat}
  The final example deals with the $\Ext$-category
  associated to a subcategory of an abelian category. The Koszul dual
  of a Koszul category, that we define in Section \ref{section2}, is
  obtained in this way.

Let $\C$ be an abelian category. For a
full subcategory $\C'$ consider the $\Ext$-category $E(\C')$ of $\C'$,
which has the same objects as $\C'$ and the homomorphisms are given by
\[\Hom_{E(\C')}(A,B)=\amalg_{i\geq 0}\Ext^i_\C(A,B)\]
for all objects $A$ and $B$ in $E(\C')$.  Then $E(\C')$ is a positively
graded category with
\[\Hom_{E(\C')}(A,B)_i=\Ext^i_\C(A,B)\]
for all objects $A$ and $B$ in $E(\C')$.

Here again one of the most important example and application for us is
the category $\C=\Mod(\mod\L)$ for a finite dimensional algebra $\L$
and $\C'$ the full subcategory consisting of all simple functors.
\end{example}

\begin{remark}
For both applications to finite dimensional algebras in Example
\ref{exam:assgraded} and Example \ref{exam:extcat} the graded parts of
the categories in question are always semisimple.
\end{remark}

\subsection{Functor categories of graded categories}

Next we discuss functors between graded categories. Let $\C$ and $\D$
be two graded $K$-categories. A covariant functor $F\colon \C\to \D$
of graded categories is a functor between the (ungraded) categories
$\C$ and $\D$ such that $F$ induces a degree zero homomorphism of the
$\Z$-graded vector spaces $\Hom_\C(C,D)$ and $\Hom_\D(F(C),F(D))$;
that is,
\begin{multline}
F\colon \Hom_\C(C,D)=\amalg_{i\in\Z} \Hom_\C(C,D)_i\to\notag\\
\amalg_{i\in\Z}\Hom_\C(F(C),F(D))_i = \Hom_\D(F(C),F(D))\notag
\end{multline}
is a degree zero homomorphism.  A contravariant functor between graded
$K$-categories is defined similarly.

\begin{example}
Let $\C$ be a graded $K$-category. For an object $C$ in $\C$ the
representable functors $\Hom_\C(-,C)\colon \C^\op\to \Gr(K)$ and
$\Hom_\C(C,-)\colon \C\to \Gr(K)$ are covariant functors from the
graded $K$-categories $\C^\op$ and $\C$ into the category of graded
$K$-vector spaces, respectively.
\end{example}

Let $\C$ be an additive graded $K$-category. Denote by $\Gr(\C)_0$ the
category having as objects the additive graded functors $F\colon
\C^\op\to \Gr(K)$ and morphisms being the natural transformations
$\eta\colon F\to G$ of $F$ and $G$, where $\eta_C\colon F(C)\to G(C)$
is a degree zero homomorphism for each object $C$ in $\C$. This is an
abelian category.

Given $F$ in $\Gr(\C)_0$ we define a shift operation $[j]$ for any
integer $j$ on the functor $F$ by letting
\[(F[j])(C)=F(C)[j]\]
and
\[(F[j])(f)=F(f)[j]\]
for any objects $C$ and $D$ in $\C$ and any map $f\colon C\to D$. In
other words $F[j]=[j]\comp F$, where the last $[j]$ denotes the $j$-th
shift operator in $\Gr(K)$.

We define the category $\Gr(\C)$ as the category with the same objects
as $\Gr(\C)_0$ and morphisms given by
\[\Hom_{\Gr(\C)}(F,G)=\amalg_{i\in\Z}\Hom_{\Gr(\C)_0}(F,G[i])\]
for all objects $F$ and $G$ in $\Gr(\C)$.  In this way $\Gr(\C)$
becomes a graded $K$-category.

\subsection{Yoneda's Lemma} The Yoneda's Lemma is fundamental in the
theory of functors. Here we give a graded version, and the proof is
given for completeness.

\begin{lem}\label{lem:Yoneda}
Let $\C$ be a graded $K$-category. The morphism
\[\alpha\colon \Hom_{\Gr(\C)}(\Hom_\C(-,C),F)\simeq F(C)\]
given by $\alpha(\eta)=\eta_C(1_C)$ for any $\eta\colon \Hom_C(-,C)\to
F$, is a degree zero isomorphism for any $C$ in $\C$ and for any $F$
in $\Gr(\C)$.
\end{lem}
\begin{proof}
Consider a homogeneous element $\psi$ of degree $j$ in
\[\Hom_{\Gr(\C)}(\Hom_\C(-,C),F)=\amalg_{i\in\Z}
\Hom_{\Gr(\C)_0}(\Hom_\C(-,C),F[i]).\]
\sloppy This means that $\psi=\{\psi_X\}_{X\in\C}$ and each $\psi_X$
is a natural transformation with $\psi_X\colon \Hom_\C(X,C)\to
F[j](X)$ a degree zero homomorphism for all $X$ in $\C$.  In
particular, $\alpha(\psi)=\psi_C(1_C)$ is in $F[j](C)_0=F(C)_j$, and
therefore $\alpha$ is a degree zero homomorphism of $K$-vector
spaces. Hence to show that $\alpha$ is an isomorphism, it is enough to
consider homogeneous elements.

Next we show that $\alpha$ is injective. Suppose that $\alpha(\psi)=0$
for a homogeneous element $\psi$ of degree $j$ in
$\Hom_{\Gr(\C)}(\Hom_\C(-,C),F)$. We want to prove that $\psi_X=0$ for
all $X$ in $\C$. Let $h=\{h_i\}_{i\in\Z}$ be in
$\Hom_\C(X,C)=\amalg_{i\in\Z}\Hom_\C(X,C)_i$. Then
we have the following commutative diagram for all $i$ in $\Z$
\[\xymatrix{%
\Hom_\C(C,C) \ar[r]^{\psi_C} \ar[d]^{\Hom_\C(h_i,C)} &
F[j](C)\ar[d]^{F[j](h_i)} \\
\Hom_\C(X,C)[i] \ar[r]^{\psi_X} & F[j](X)[i]}\]
so that $F[j](h_i)\psi_C(1_C)=\psi_X(h_i)$. Since $\psi_C(1_C)=0$, we
have that $\psi_X(h_i)=0$ for all $i$ and all $X$ in $\C$. Hence
$\psi=0$, and $\alpha$ is injective.

Finally we prove that $\alpha$ is surjective. Let $\mu$ be in
$F(C)_j$. Define $\psi=\{\psi_X\}_{X\in\C}\colon \Hom_\C(-,C)\to F[j]$
for all $X$ in $\C$ and $h=\{ h_i\}_{i\in\Z}$ in $\Hom_\C(X,C)$ by
letting
\[\psi_X(h)=(F(h_i)(\mu))_{i\in\Z}=F(h)(\mu)\]
viewing $F(h)(\mu)$ as an element in $F[j](X)$.  To see that this all
makes sense observe the following. For $h_i$ in $\Hom_\C(X,C)_i$ the
image of $h_i$ under $F$ is a map $F(h_i)\colon F(C)\to F(X)$ of
degree $i$. Hence $F(h_i)(\mu)$ is in $F(X)_{i+j}=F[j](X)_i$, and
$F(h)(\mu)$ is in $F[j](X)$. Clearly we have that
$\alpha(\psi)=\mu$. So if $\psi$ is a natural transformation the proof
is complete. To this end let $g\colon Y\to X$ be in $\C$ and consider
the diagram
\[\xymatrix{%
\Hom_\C(X,C)\ar[r]^{\psi_X}\ar[d]^{\Hom_\C(g,C)} &
F[j](X)\ar[d]^{F[j](g)} \\
\Hom_\C(Y,C)\ar[r]^{\psi_Y} & F[j](Y)}\]
For $h\colon X\to C$ in $\C$ we obtain that
\[F[j](g)\psi_X(h)=F(g)(F(h)(\mu))=F(hg)(\mu)=\psi_Y(hg)=
\psi_Y(\Hom_\C(g,C)(h)).\]
Hence $\psi$ is a natural transformation, and $\alpha$ is a degree
zero isomorphism.
\end{proof}

\subsection{Ideals}\label{subsec:ideals}
Secretly we have considered ideals in a category already as we have
discussed the radical of a category. Here we review the definition of
a (graded) ideal in a category and some elementary constructions and
results involving ideals.

Recall that an ideal in a category is a sub-bifunctor of the
$\Hom$-functor. In case $\C$ is a $\Z$-graded $K$-category, a graded
ideal in $\C$ is a graded sub-bifunctor of the $\Hom$-functor.

For two ideals $\I_1$ and $\I_2$ in a category we define in a natural
way inclusion, intersection and product. In particular, the product of
two ideals $\I_1$ and $\I_2$ is given by
\[\I_1\I_2(X,Y)=\{f\in\Hom_\C(X,Y)\mid f=\sum_{i=1}^nh_ig_i,
g_i\in\I_2(X,A_i), h_i\in\I_1(A_i,Y)\},\]
and if $\I_1$ and $\I_2$ are graded ideals, then their product
$\I_1\I_2$ is a graded ideal with
\begin{multline}
\I_1\I_2(X,Y)_t=\{f\in\Hom_\C(X,Y)_t\mid \\
f=\sum_{i=1}^nh_ig_i, g_i\in\I_2(X,A_i)_{n_i},
h_i\in\I_1(A_i,Y)_{t-n_i}\}.\notag
\end{multline}
For an $n$-fold product of an ideal $\I$ with itself we write $\I^n$.

Let $\I$ be an ideal in a category $\C$. For a functor $F\colon
\C^\op\to \Mod K$ define $\I F\colon \C^\op\to \Mod K$ as the
subfunctor of $F$ given by
\[\I F(M)=\sum_{\substack{f\in\I(M,X)\\ X\in \C}} \Im F(f)\]
for all objects $M$ in $\C$.  If $\I$ is a graded ideal and $F\colon
\C^\op\to \Gr(K)$ is a graded functor, $\I F$ is a graded subfunctor
of $F$, where
\[\I F(M)_n=\sum_{\substack{f\in\I(M,X)_i\\ X\in\C\\ i+j=n}}
F(f)(F(X)_j)\] for all objects $M$ in $\C$.  Easy properties of the
product of an ideal and a functor are the following.
\begin{lem}\label{lem:idealepi}
Let $\I$ be an ideal and $\eta\colon F\to G$ a morphism of two
functors $F$ and $G$ in $\Mod(\C)$.
\begin{enumerate}
\item[(a)] $\eta(\I F)\subseteq \I G$.
\item[(b)] If $\eta\colon F\to G$ is an epimorphism, then
  $\eta|_{\I F}\colon \I F\to \I G$ is an epimorphism.
\end{enumerate}
\end{lem}

Given an ideal in an additive $K$-category $\C$ there is a naturally
associated graded category. We saw an example of this already in
Example \ref{exam:assgraded}. Let $\C$ be an additive $K$-category
with an ideal $\I$. Denote by $\A_\gr(\C)$ the associated graded
category with respect to $\I$ having the same objects as $\C$ while
the morphisms are given by
\[\Hom_{\A_\gr(\C)}(A,B)=\amalg_{j\geq 0} \I^j(A,B)/\I^{j+1}(A,B)\]
for all objects $A$ and $B$ in $\A_\gr(\C)$.

For each object and each morphism in $\Mod(\C)$ there are naturally
associated an object and a morphism in $\Gr(\A_\gr(\C))$. Denote this
function on objects and morphisms by $G$, and $G\colon \Mod(\C)\to
\Gr(\A_\gr(\C))$ is given by letting
\[G(F)=\amalg_{j\geq 0}\I^j F/\I^{j+1} F\]
for all $F$ in $\Mod(\C)$, where $G(F)_j=\I^jF/\I^{j+1}F$. On a morphism
$\eta\colon F\to H$ in $\Mod(\C)$ let
\[G(\eta)=(\overline{\eta}|_{\I^j F/\I^{j+1} F})_{j\geq 0} \colon
G(F)\to G(H).\]
An ideal $\I$ gives a filtration of any object $F$ in $\Mod(\C)$ via
$\{\I^jF\}_{j\geq 0}$. Any natural transformation $\eta\colon F\to F'$
for $F$ and $F'$ in $\Mod(\C)$ we have that $\eta(\I^jF)\subseteq
\I^jF'$ for all $j\geq 0$ by  Lemma \ref{lem:idealepi}. So any object
in $\Mod(\C)$ has a filtration and any morphism in $\Mod(\C)$ has
degree $0$, referring to \cite[I.2]{NvO}. Then by \cite[I.4]{NvO} the
function $G$ is a functor $G\colon \Mod(\C)\to \Gr(\A_\gr(\C))$. 

The notion of an ideal in a graded category leads to the following
natural definition of a graded category generated in degrees $0$ and
$1$.
\begin{defin}
Let $\C$ be a positively graded $K$-category. The graded category $\C$
is said to be \emph{generated in degrees $0$ and $1$} if the ideal
\[J=\amalg_{i\geq 1}\Hom_\C(-,-)_i\subseteq \Hom_\C(-,-)\]
satisfies
\[J^r=\amalg_{i\geq r}\Hom_\C(-,-)_i\]
for all $r\geq 1$.
\end{defin}

\subsection{Tensor product of functors}  Tensor products of functors
were first considered by Mitchell in \cite{M} and then later by
Auslander and Bautista et.\ al.\ in \cite{A,BCS}. Here we review the
construction of tensor products of functors from \cite{A}.

Let $\C$ be an additive $K$-category. For two functors $F$ in
$\Mod(\C)$ and $G$ in $\Mod(\C^\op)$ we want to define the tensor
product $G\otimes_\C F$ of $G$ and $F$. There is a unique (up to
isomorphism) functor $-\otimes_\C-\colon \Mod(\C^\op)\times\Mod(\C)
\to \Mod K$ satisfying the following properties:
\begin{enumerate}
\item[(i)] The tensor product is a right exact functor in each
 variable.
\item[(ii)] The tensor product commutes with direct sums in both
variables.
\item[(iii)] For each object $C$ in $\C$ we have
 $\Hom_\C(C,-)\otimes_\C F = F(C)$ and
 $G\otimes_\C\Hom_\C(-,C)=G(C)$ for any $F$ in $\Mod(\C)$ and $G$ in
 $\Mod(\C^\op)$.
\end{enumerate}
A morphism $\amalg_{j}\Hom_\C(Y_j,-)\to \amalg_i\Hom_\C(X_i,-)$ is
given by morphisms $f_{ij}\colon X_i\to Y_j$ such that $(f_{ij})$ is in
$\prod_j\amalg_i\Hom_\C(X_i,Y_j)$, and
\[(\Hom_\C(f_{ij},-))\otimes 1_F\colon
\amalg_j\Hom_\C(Y_j,-)\otimes_\C F\to \amalg_i\Hom_\C(X_i,-)\otimes_\C
F\]
is by definition given by $(F(f_{ij}))\colon \amalg_j F(Y_j)\to
\amalg_i F(X_i)$. 

Suppose that $G$ is in $\Mod\C$, and let
\[\amalg_j\Hom_\C(Y_j,-)\extto{(f_{ij},-)} \amalg_i\Hom_\C(X_i,-)\to
G\to 0\]
be a projective presentation of $G$ (See Section \ref{section1.5} for
a further discussion on projective functors). Then $G\otimes_\C F$ is
by definition given by the commutative diagram
\[\xymatrix{%
\amalg_j\Hom_\C(Y_j,-)\otimes_\C F\ar[r] \ar@{=}[d] &
\amalg_i\Hom_\C(X_i,-)\otimes_\C F\ar[r] \ar@{=}[d] &
G\otimes_\C F\ar[r]\ar@{=}[d] & 0\\
\amalg_j F(Y_j)\ar[r]^{(F(f_{ij}))} & \amalg_i F(X_i)\ar[r] &
G\otimes_\C F\ar[r] & 0}\]
It is straightforward to see that the definition of $G\otimes_\C F$ is
independent of the chosen projective presentation of $G$. Furthermore,
it follows that the tensor product $-\otimes_\C-$ is a right exact
functor in both variables. An equivalent definition is to use a
projective presentation of $F$ to define $G\otimes_\C F$. We freely
identify these.

Let $B\colon \C^\op\times\C\to \Mod K$ be an additive bifunctor, and let
$F$ be in $\Mod(\C)$. Then define $B\otimes_\C F\colon \C^\op\to
\Mod K$ by
\[(B\otimes_\C F)(X) = B(X,-)\otimes_\C F\]
and for $f\colon X\to Y$ in $\C$
\[(B\otimes_\C F)(f)\colon B(Y,-)\otimes_\C F\extto{B(f,-)\otimes_\C
1_F} B(X,-)\otimes_\C F.\]
Applying this to an ideal in an additive $K$-category $\C$ we obtain the
following.
\begin{lem}\label{lem:tensormoduloideal}
Let $\I$ be an ideal in an additive $K$-category $\C$. Then
\[\Hom_\C(-,-)/\I\otimes_\C F\simeq F/\I F\]
for all $F$ in $\Mod(\C)$.
\end{lem}
\begin{proof}
The exact sequence
\[0\to \I\to \Hom_\C(-,-)\to \Hom_\C(-,-)/\I\to 0\]
gives rise to the exact sequence
\[\I\otimes_\C F \to \Hom_\C(-,-)\otimes_\C F\to
\Hom_\C(-,-)/\I\otimes_\C F\to 0.\]
For an object $X$ in $\C$ we obtain
\[\xymatrix@C=15pt{%
& \I(X,-)\otimes_\C F\ar[r]^-{\theta_X} \ar[d] &
\Hom_\C(X,-)\otimes_\C F\ar[r]\ar@{=}[d] &
\Hom_\C(X,-)/\I(X,-)\otimes_\C F\ar[r]\ar[d]^-\wr & 0\\
0\ar[r] & \Im \theta_X\ar[r] & F(X)\ar[r] & F(X)/\Im
\theta_X \ar[r] & 0}\]
We want to show that $\Im \theta_X=(\I F)(X)$ for all $X$ in $\C$.

Let $\amalg_j\Hom_\C(-,B_j)\to \amalg_i\Hom_\C(-,C_i)\extto{\varphi}
F\to 0$ be a projective presentation of $F$. Recall that by Yoneda's
Lemma $\varphi$ is given as follows. Let $x_i=\varphi(1_{C_i})$ in
$F(C_i)$. For $(f_i)_i$ in $\amalg_i(X,C_i)$ we have
$\varphi((f_i)_i)=\sum_i F(f_i)(x_i)$.

Consider the following commutative diagram
\[\xymatrix@C=10pt{
& \amalg_i\Hom_\C(X,C_i)\ar@{=}[r] \ar[d]^{\varphi_X} & 
\Hom_\C(X,-)\otimes_\C \amalg_i\Hom_\C(-,C_i)\ar[d]^{\pi_X}\\
& F(X)\ar@{=}[r]|!{[d];[ur]}\hole  &  
        \Hom_\C(X,-)\otimes_\C F\\ 
\amalg_i\I(X,C_i) \ar@{=}[r]\ar[d]\ar[uur]^(0.3){\nu_X} & 
\I(X,-)\otimes_\C\amalg_i\Hom_\C(-,C_i)\ar[d]\ar[uur]^(0.3){\theta_X} & \\
\I(X,-)\otimes_\C F\ar@{=}[r]\ar[uur]|!{[u];[ur]}{\hole\hole}^(0.7){\nu_X} & 
        \I(X,-)\otimes_\C F\ar[uur] & 
}\]
It is clear from this diagram that
$\Im\theta_X=\Im\pi_X\nu_X=\Im\varphi_X\nu_X$. Given $(f_i)_i$ in
$\amalg_i\I(X,C_i)$ we have that
$\varphi_X\nu_X((f_i)_i)=\varphi_X((f_i)_i)=\sum_i F(f_i)(x_i)$ which
is in $(\I F)(X) = \sum_{\substack{h\in\I(X,Y)\\ Y\in\C}} F(h)$. Hence
$\Im\theta_X\subseteq (\I F)(X)$.

Conversely, let $y$ be in $(\I F)(X)$. By definition there exists some
$g$ in $\I(X,Y)$ and $z$ in $F(Y)$ such that $y=F(g)(z)$. By Yoneda's
Lemma $z$ corresponds to a map $\psi_z\colon (-,Y)\to F$. Then the
diagram
\[\xymatrix{%
& \Hom_\C(-,Y)\ar[d]^{\psi_z}\ar@{-->}[dl]_{\exists (-,f_i)_i} \\
\amalg_i{\Hom_\C(-,C_i)}\ar[r]^-\varphi & F}\]
gives a map $(f_i)_i$ in $\amalg_i{_\C(Y,C_i)}$ such that
$\psi_z(1_Y)=z=\sum_iF(f_i)(x_i)$. Applying $F(g)$ to this equality we
have that
\[y=F(g)(z)=F(g)\sum_iF(f_i)(x_i)=\sum_iF(f_ig)(x_i),\]
which is in $\Im\theta_X$. Hence $\Im\theta_X=(\I F)(X)$. It follows
that
\[((-,-)/\I(-,-)\otimes_\C F)(X)\simeq F(X)/\Im\theta_X = F(X)/\I
F(X)= (F/\I F)(X).\]
Given a map $h\colon Y\to X$ we have a commutative diagram
\[\xymatrix{%
\I(X,-)\otimes_\C F\ar[r]^-{\theta_X}\ar[d]^{(h,-)|_{\I(X,-)}\otimes 1_F} &
\Hom_\C(X,-)\otimes_\C F \ar@{=}[r]\ar[d]^{(h,-)\otimes 1_F} &
F(X)\ar[d]^{F(h)} \\
\I(Y,-)\otimes_\C F\ar[r]^-{\theta_Y} & \Hom_\C(Y,-)\otimes_\C F \ar@{=}[r]
& F(Y) }\]
Hence $\Im F(h)\theta_X\subseteq \Im\theta_Y$. Let $z$ be in $(\I
F)(X)$, then $z=F(g)(w)$ for some $g$ in $\I(X,W)$ and $w$ in
$F(W)$. For all $h$ in $\Hom_\C(Y,X)$ the composition $gh$ is in
$\I(Y,W)$. Furthermore $F(gh)(w)=F(h)F(g)(w)=F(h)(z)$, and it is in
$(\I F)(Y)$. Therefore $\Im\theta=\I F$ as functors. This completes
the proof.
\end{proof}

Let $\C$ be an additive graded $K$-category. Here we extend the
definition of tensor products of functors to tensor products of graded
functors and obtain similar results.

For two graded functors $F$ in $\Gr(\C)$ and $G$ in $\Gr(\C^\op)$,
then choosing a graded projective presentation
\[\amalg_j\Hom_\C(Y_j,-)[n_j]\extto{(f_{ij})}
\amalg_i\Hom_\C(X_i,-)[m_i]\to G\to 0\]
of $G$ one similarly defines $G\otimes_\C F$ as
\[\amalg_j F(Y_j)[n_j]\extto{(F(f_{ij}))}\amalg_i F(X_i)[m_i]\to
G\otimes_\C F\to 0.\]
Since $\amalg_j F(Y_j)[n_j]\extto{(F(f_{ij}))}\amalg_i F(X_i)[m_i]$ is
a degree zero map of graded vector spaces, $G\otimes_\C F$ naturally
becomes a graded vector space. Again this is
independent of the chosen projective presentation of $G$ (or the
one of $F$). 

\subsection{Nakayama's Lemma} As for ring theory in general
Nakayama's Lemma is a central result. Here we give a version for
graded functors. To this end we need the following definition.

Let $\C$ be a graded $K$-category. Then a graded functor $F$ in
$\Gr(\C)$ is said to be \emph{bounded below} if $F(C)_i=(0)$ for all
objects $C$ in $\C$ and $i<N$ for some integer $N$.

\begin{lem}[Nakayama's Lemma]
  Let $\C$ be an additive graded $K$-category with radical $\rad_\C
  =\amalg_{i\geq 1} \Hom_\C(-,-)_i$.  Suppose that $F$ in $\Gr(\C)$ is
  a bounded below graded functor. Assume that $F/\rad_\C F=(0)$, then
  $F=0$.
\end{lem}
\begin{proof}
Assume that $F$ is bounded below such that $F(C)_i=(0)$ for all $i<N$
and all $C$ in $\C$ and that $F(C)_N\neq (0)$ for some $C$ in
$\C$. By definition
\[(\rad_\C F)(C)=\sum_{\substack{g\in\rad_\C(C,X)\\ X\in \C}} \Im
F(g)\subseteq F(C).\]
Since $\rad_\C(C,X)=\amalg_{i\geq 1}\Hom_\C(C,X)_i$, then for any $g$ in
$\rad_\C(C,X)$ the morphism $F(g)\colon F(X)\to F(C)$ might be of mixed
degree but always of degree greater or equal to $1$. Since
$F(X)_i=(0)$ for all $i<N$, we have $\Im F(g)\subseteq \amalg_{i>N}
F(C)_i$. It follows that $F/\rad_\C F\neq 0$.
\end{proof}
The assumption on the radical of the category $\C$ above seems
strong. However, the applications we have in mind are coming from
Example \ref{exam:assgraded} and Example \ref{exam:extcat}. In
addition, the situation we want to generalize is that of a positively
graded algebra $\L=\sum_{i\geq 0}\L_i$ with $\L_0$ semisimple.

\section{Homological algebra}\label{section1.5}

The main aim in this section is to prove some elementary homological
properties we use later for the graded categories discussed in the
previous section. When $\C$ is a positively graded Krull-Schmidt
category, we characterize the projective and the simple objects in
$\Gr(\C)$, discuss a duality for subcategories of $\Gr(\C)$ and show
that the global dimension of $\Gr(\C)$ is given by the supremum of the
projective dimension of the simple objects.

\subsection{Projective functors and covers}
Using the Yoneda's Lemma it is well-known that the functors
$\Hom_\C(-,C)$ are projective in $\Mod(\C)$ and $\Gr(\C)$ for an
additive graded $K$-category $\C$. In addition any projective functor
is a direct summand of $\amalg_{i\in \Sigma} \Hom_\C(-,C_i)[m_i]$ for
some integers $m_i$ and index set $\Sigma$. By further assuming that
the category $\C$ is Krull-Schmidt, we show next that all projective
functors in $\Gr(\C)$ are exactly given like this for indecomposable
objects $C_i$.

\begin{lem}
Let $\C$ be a graded Krull-Schmidt $K$-category.  The projective
functors in $\Gr(\C)$ are all of the form
\[\amalg_{j\in J}\Hom_\C(-,C_j)[m_j]\] with $C_j$ indecomposable in
$\C$ and $m_j$ an integer for some index set $J$.
\end{lem}
\begin{proof}
Let $F$ be a projective functor in $\Gr(\C)$. Then there exists an
exact sequence $\amalg_{j\in I}\Hom_\C(-,C_j)[m_j]\extto{\pi} F$ for
some degree zero morphism $\pi$ and for some indecomposable objects
$C_j$ in $\C$, integers $m_j$ and some index set $J$. Since $F$ is
projective in $\Gr(\C)$, there is a splitting of the morphism $\pi$
and this can be chosen also as a degree zero morphism. Standard
arguments then shows there is a degree $0$ isomorphism
\[\varphi\colon \amalg_{j\in I}\Hom_\C(-,C_j)[m_j] \to F\amalg G\]
for some functor $G$ in $\Gr(\C)$ (choose $G=\Ker \pi$).

Consider the composition of the maps
\begin{multline}
\Hom_\C(-,C_j)[m_j]\extto{\lambda_j} \amalg_{j\in
  I}\Hom_\C(-,C_j)[m_j] \extto{p_F\varphi} F\extto{\varphi^{-1} i_F}\\
\amalg_{j\in I}\Hom_\C(-,C_j)[m_j] \extto{p_j}
  \Hom_\C(-,C_j)[m_j],\notag
\end{multline}
where $p_F$ and $i_F$ are the natural projection and inclusion of $F$
in $F\amalg G$, respectively. Then
$p_j\varphi^{-1}i_Fp_F\varphi\lambda_j\colon \Hom_\C(-,C_j)[m_j]\to
\Hom_\C(-,C_j)[m_j]$ is a natural transformation of degree $0$ induced
by a map $g\colon C_j\to C_j$ in $\Hom_\C(C_j,C_j)_0$ by the Yoneda's 
Lemma. Since 
$\Hom_\C(C_j,C_j)_0$ is a local ring and $1_{F\amalg G}=
i_Fp_F+i_Gp_G$, either $g$ or $1-g$ is an isomorphism.

If $g$ is an isomorphism, then $\Hom_\C(-,C_j)[m_j]$ is a direct
summand of $F$. And if $1-g$ is an isomorphism, then
$\Hom_\C(-,C_j)[m_j]$ is a direct summand of $G$.  Hence, either
$\Hom_\C(-,C_j)[m_j]$ is a direct summand of $F$ or a direct summand
of $G$.

Now one can proceed as in the proof of \cite[Theorem
  26.5]{AF}. Consider pairs $(J,L)$ of subsets of $I$ with $J\cap
L=\emptyset$ and such that $\amalg_{j\in J}\Hom_\C(-,C_j)[n_j]$ is a
subfunctor of $F$ and $\amalg_{l\in L}\Hom_\C(-,C_l)[n_l]$ is
subfunctor of $G$. These pairs $(J,L)$ is naturally ordered by
inclusion, and any chain has a upper bound given by the union. So, by
Zorn's Lemma we can choose a maximal pair $(J,L)$. Then we have the
following commutative diagram
\[\xymatrix@C=15pt{
0\ar[d] & 0\ar[d] \\
\amalg_{j\in J\cup L}\Hom_\C(-,C_j)[n_j] \ar@{=}[r]\ar[d] &
(\amalg_{j\in J}\Hom_\C(-,C_j)[n_j])\amalg
(\amalg_{l\in L}\Hom_\C(-,C_l)[n_l])\ar[d] \\
\amalg_{i\in I}\Hom_\C(-,C_i)[n_i] \ar[r]^-\sim\ar[d] &
F\amalg G\ar[d] \\
\amalg_{i\in I\setminus (J\cup L)}\Hom_\C(-,C_i)[n_i] \ar[r]^-\sim\ar[d] &
F'\amalg G'\ar[d]\\
0 & 0
}\]
where $F'=F/(\amalg_{j\in J}\Hom_\C(-,C_j)[n_j])$ and
$G'=G/(\amalg_{l\in L}\Hom_\C(-,C_l)[n_l])$. Hence, any
$\Hom_\C(-,C_{j_0})[n_{j_0}]$ with $j_0$ in $I\setminus (J\cup L)$ is
either a direct summand of $F'$ or a direct summand of $G'$. It
follows that either
\[(\amalg_{j\in J}\Hom_\C(-,C_j)[n_j])\amalg
\Hom_\C(-,C_{j_0})[n_{j_0}]\]
is a subfunctor of $F$ or
\[(\amalg_{l\in L}\Hom_\C(-,C_l)[n_l])\amalg
\Hom_\C(-,C_{j_0})[n_{j_0}]\]
is a subfunctor of $G$. This contradicts the choice of the maximal pair
$(J,L)$. The claim follows from this.
\end{proof}

Next we discuss projective covers in the category of functors we are
considering. Recall that an essential epimorphism $P\to F$ in
$\Mod(\C)$ is a \emph{projective cover} of $F$ if $P$ is a projective
$\C$-module.  For the category of modules over a ring, all simple
modules have a projective cover if and only if the ring is
semiperfect. It was shown by Auslander in \cite{A} that the same
condition comes up in having minimal projective presentations of
finitely presented functors as we recall next. Denote by $\mod\C$ the
full subcategory of $\Mod(\C)$ consisting of all finitely presented
functors.

\begin{lem}[\protect{\cite[Corollary 4.13]{A}}]\label{lem:minpresent}
Let $\C$ be a positively graded $K$-category, where (all idempotents
split and) $\End_\C(C)_0$ is semiperfect for all objects $C$ in
$\C$. Then every object in $\mod(\C)$ has a minimal projective
presentation. In particular, any object in $\mod(\C)$ has a projective
cover.
\end{lem}
By our remark in the introduction, for a graded Krull-Schmidt
$K$-category $\C$ the category of finitely presented functors $\mod\C$
has minimal projective presentations.

There is another situation where projective covers always exists. To
motivate this recall the following situation for graded algebras. Let
$\L=\oplus_{i\geq 0}\L_i$ be a positively graded $K$-algebra with
$\L_0$ semisimple. Let $M$ be a graded module bounded below, that is,
$M$ is generated in some degrees $i_0<i_1<i_2<\cdots$. Then $M$ has a
projective cover. An analogue for graded functors is the following,
which is also related to the above version of the Nakayama's Lemma.

\begin{lem}\label{lem:posgradprojcover}
\sloppy Let $\C$ be a positively graded Krull-Schmidt $K$-category
with $\rad_{\C}(-,-)=\amalg_{i\geq 1}\Hom_\C(-,-)_i$.
\begin{enumerate}
\item[(a)] Any bounded below functor $F$ in $\Gr(\C)$ has a projective
  cover.
\item[(b)] Let $F$ in $\Gr(\C)$ be bounded below, and let $P\to F$ be
  a projective cover. Then $P/\rad_\C P\simeq F/\rad_\C F$.
\end{enumerate}
\end{lem}
\begin{proof}
(a) Since any bounded below functor $F$ in $\Gr(\L)$ is a factor of
shifts of copies of $\Hom_\C(-,C)$ for $C$ indecomposable in $\C$, it
follows that $F/\rad_\C F$ is a direct sum of shifts of $t_i$ copies of
simple functors $(-,C_i)/\rad_\C(-,C_i)$ for some integers $t_i$, say
$\amalg_{i\in I}\left((-,C_i)/\rad_\C(-,C_i)[n_i]\right)^{t_i}$. Then we 
obtain an induced morphism $\pi\colon \amalg_{i\in
I}\left((-,C_i)[n_i]\right)^{t_i}\to F$. Using Nakayama's Lemma we
infer that $\pi$ is an epimorphism. Let $P=\amalg_{i\in
  I}(-,C_i)[n_i]^{t_i}$. Then we have that $P/\rad_\C P\simeq
F/\rad_\C F$ and using Nakayama's Lemma again we see that $\pi$ is an
essential epimorphism. Hence $\pi\colon P\to F$ is a projective cover.

(b) The claim follows from the construction in (a).
\end{proof}

\subsection{Simple functors} In Koszul theory simple modules play
a crucial role. So there is no surprise in generalizing to functor
categories that the simple functors are of equally great
importance. Here we show that for a positively graded Krull-Schmidt
$K$-category, there is one-to-one correspondence between
indecomposable objects and simple functors.

To show the above claim we shall need the following considerations.
Given a graded module $M=\amalg_{i\in\mathbb{Z}} M_i$ over some
positively graded ring, the set $M_{\geq n} = \amalg_{i\geq n} M_i$ is
a graded submodule of $M$. There is a similar construction for graded
functors. Let $\C$ be a positively graded $K$-category, and let $F$ be
in $\Gr(\C)$. Then we define $F_{\geq n}$ as
\[F_{\geq n}(C) = (F(C))_{\geq n}\]
for all $C$ in $\C$, and for $f\colon C\to C'$ in $\C$
\[F_{\geq n}(f)=F(f)|_{F_{\geq n}(C)}\colon F_{\geq n}(C)\to F_{\geq
  n}(C').\]
Since $\C$ is positively graded, we infer that $F_{\geq n}$ is a
graded subfunctor of $F$. As a consequence of this we obtain that a
simple object $S$ in $\Gr(\C)$ is supported only in one degree; that
is, $S(C)_i\neq (0)$ for one fixed $i=i_0$ for all objects $C$ in
$\C$.

When $\C$ in addition is Krull-Schmidt, we obtain even more as shown
next.
\begin{lem}\label{lem:simples}
\sloppy Let $\C$ be a positively graded Krull-Schmidt $K$-category, and let
  $\rad_\C(-,-)$ be the radical of $\C$.
\begin{enumerate}
\item[(a)] Any simple functor in $\Gr(\C)$ is of the form
  $\Hom_\C(-,C)/\rad_\C(-,C)$ for some indecomposable object $C$ in
  $\C$ up to shift.
\item[(b)] For all finitely generated functors $F$ in $\Gr(\C)$ the
  radical of $F$ is given by $\rad_\C F$.
\item[(c)] If $F$ is a finitely generated functor in $\Gr(\C)$ with
  $F=\rad_\C F$, then $F=0$.
\item[(d)] All finitely generated functors $F$ in $\Gr(\C)$ have a
  projective cover.
\end{enumerate}
\end{lem}
The proof of this result is basically the same as the following
result, which we give a proof of.

\begin{lem}
Let $\C$ be a Krull-Schmidt $K$-category, and let $\rad_\C(-,-)$ be
the radical of $\C$.
\begin{enumerate}
\item[(a)] Any simple functor in $\Mod(\C)$ is of the form
  $\Hom_\C(-,C)/\rad_\C(-,C)$ for some indecomposable object $C$ in
  $\C$.
\item[(b)] For all finitely generated functors $F$ in $\Mod(\C)$ the
  radical of $F$ is given by $\rad_\C F$.
\item[(c)] If $F$ is a finitely generated functor in $\Mod(\C)$ with
  $F=\rad_\C F$, then $F=0$.
\item[(d)] All finitely generated functors $F$ in $\Mod(\C)$ have a
  projective cover.
\end{enumerate}
\end{lem}
\begin{proof}
  (a) Let $S$ be a simple functor in $\Mod(\C)$. Then for some
  indecomposable object $C$ in $\C$, the vector space $S(C)$ is
  non-zero. By Yoneda's Lemma there exists a non-zero morphism
  $\eta\colon {\Hom_\C(-,C)}\to S$, which necessarily is an
  epimorphism. We claim that this is a projective cover of $S$.

  First we show that for a finitely generated functor $F$ in
  $\Mod(\C)$ and an epimorphism $\eta'\colon {\Hom_\C(-,X)}\to F$, the
  morphism $\eta'$ is minimal if and only if $\eta'$ is an essential
  epimorphism. Assume that $\eta'$ is minimal, and let $\rho\colon
  H\to {\Hom_\C(-,X)}$ be such that $\eta'\rho\colon H\to F$ is an
  epimorphism. Since $\Hom_\C(-,X)$ is projective, there exists a morphism
  $\sigma\colon {\Hom_\C(-,X)}\to H$ such that $\eta'\rho\sigma =
  \eta'$. Since $\eta'$ is minimal, $\rho\sigma$ is an isomorphism,
  and in particular $\rho$ is an epimorphism. This shows that $\eta'$
  is an essential epimorphism.

  Conversely, assume that $\eta'$ is an essential epimorphism, and let
  $\gamma\colon {\Hom_\C(-,X)}\to {\Hom_\C(-,X)}$ be such that $\eta'\gamma=
  \eta'$. It follows that $\gamma$ is an epimorphism, and since
  $\Hom_\C(-,X)$ is projective, there exists a morphism $\sigma\colon
  {\Hom_\C(-,X)}\to {\Hom_\C(-,X)}$ such that $\gamma\sigma =
  1_{\Hom_\C(-,X)}$. Then $\eta'\sigma = \eta'$, and as for $\gamma$ the
  morphism $\sigma$ is an epimorphism. It follows that $\sigma$ is an
  isomorphism. Hence also $\gamma$ is an isomorphism and $\eta'$ is
  minimal. This completes the proof of the above claim.

  Return to the morphism $\eta\colon {\Hom_\C(-,C)}\to S$ above. If $X$ is
  indecomposable, then 
\[\rad_\C(X,C)=\begin{cases}
    \Hom_\C(X,C), & \text{for\ } X\not\simeq C,\\
    \rad\End_\C(C), & \text{for\ } X\simeq C.
  \end{cases}\] 
We infer that $\Ker \eta \subseteq \rad_\C(-,C)$ and
  that $\Ker\eta = \rad_\C(-,C)$, since $\Ker\eta$ is a maximal
  subfunctor. Hence $S\simeq \Hom_\C(-,C)/\rad_\C(-,C)$. It remains to
  show that $\eta$ is minimal (essential epimorphism). Let
  $\gamma\colon {\Hom_\C(-,C)}\to {\Hom_\C(-,C)}$ be a natural transformation
  such that $\eta\gamma = \eta$. By Yoneda's Lemma $\gamma$ is given
  as $\Hom_\C(-,h)$ for some $h$ in $\End_\C(C)$. If $h$ is not an
  isomorphism, then $h$ is in the radical of the local ring
  $\End_\C(C)$, by our assumptions on $\C$. Hence $1-h$ is invertible
  or equivalently an isomorphism. This implies that $\eta=0$, which is
  a contradiction. Therefore $h$ is an isomorphism and $\eta\colon
  \Hom_\C(-,C)\to S$ is a projective cover.

  (b)\sloppy\ Now let $F$ be finitely generated in $\Mod(\C)$ with
  $\eta\colon {\Hom_\C(-,B)}\to F$ being an epimorphism. Then
  $\eta(\rad_\C(-,B))=\rad_\C F$ and $\Hom_\C(-,B)/\rad_\C(-,B)\to
  F/\rad_\C F$ is an epimorphism. Therefore $F/\rad_\C F$ is
  semisimple and $\rad(F/\rad_\C F)=(0)$. Moreover,
  $\eta(\rad(-,B))\subseteq \rad F$. Since $\C$ is Krull-Schmidt,
  $\rad(-,B)= \rad_\C(-,B)$ and therefore $\rad_\C F\subseteq \rad
  F$. It follows that $\rad(F/\rad_\C F)=\rad F/\rad_\C F$ and
  consequently $\rad F = \rad_\C F$.

  (c) Let $F$ be finitely generated in $\Mod(\C)$ with $F=\rad_\C F$
  and $\eta\colon {\Hom_\C(-,B)}\to F$ being an epimorphism. From (b) we
  have that $\eta(\rad\End_\C(B))=\rad_\C F(B)$. Since $\eta_B\colon
  \Hom_\C(B,B)\to F(B)$ is a morphism of $\End_\C(B)$-modules, $\rad_\C
  F(B)= (\rad\End_\C(B))F(B)$. By Nakayama's Lemma $F(B)=(0)$ and
  therefore $F=0$.

  (d) Keeping the notation and assumptions from (b), we have
  \[F/\rad_\C F\simeq \Hom_\C(-,C)/\rad_\C(-,C)\] for some $C$ in
  $\C$. This gives rise to a morphism $\eta\colon {\Hom_\C(-,C)}\to F$
  with \[\overline{\eta}\colon {\Hom_\C(-,C)}/\rad_\C(-,C)\to F/\rad_\C
  F\] an isomorphism. Let $\gamma\colon H\to {\Hom_\C(-,C)}$ such that
  $\eta\gamma$ is an epimorphism. Then $\Hom_\C(-,C)/\Im\gamma =
  \rad_\C({\Hom_\C(-,C)}/\Im\gamma)$, so that by (c) $\gamma$
  is an epimorphism and $\eta\colon {\Hom_\C(-,C)}\to F$ is a projective
  cover.
\end{proof}

\subsection{Duality}\sloppy For finite dimensional algebras, the vector
space duality $D=\Hom_K(-,K)$ provides a bridge between left and right
finite dimensional modules. For graded $K$-algebras and graded modules
$M=\amalg_{i\in \Z}M_i$ over such, each graded part $M_i$ is a vector
space over $K$. Then one defines the graded dual of $M$ as
$\amalg_{i\in \Z} D(M_{-i})$, where the degree $i$ part is
$D(M_{-i})$. A similar construction can be carried out for graded
functors, and in the following result we describe this construction
and some elementary properties and consequences of having such a
functor.

\begin{prop}\label{prop:duality}
Let $\C$ be a graded $K$-category. Then there exists a contravariant
functor
\[D\colon \Gr(\C)\to \Gr(\C^\op)\]
defined by
\[D(F)(X)=\amalg_{j\in\Z}\Hom_K(F(X)_j,K)\]
for $X$ in $\C$, where $D(F)(X)_i=\Hom_K(F(X)_{-i},K)$ for $i$ in
$\Z$.
Then the following statements hold:
\begin{enumerate}
\item[(a)] For a set of functors $\{F_i\}_{i\in I}$ we have that
$D(\amalg_{i\in I}F_i)=\prod_{i\in I}D(F_i)$.
\item[(b)] There exists an injective natural transformation
$\eta\colon 1_{\Gr(\C)}\to D^2$. If $F$ is locally finite, then
$\eta_F$ is an isomorphism of functors. In particular, if $\lfGr(\C)$
denotes the full subcategory of $\Gr(\C)$ consisting of locally
finite graded contravariant functors, then $D$ induces a duality
\[D\colon \lfGr(\C)\to\lfGr(\C^\op).\]
\item[(c)] If $B\colon \C^\op\times\C\to \Gr(K)$ is a graded
bifunctor, $G$ is in $\Gr(\C^\op)$ and $F$ is in $\Gr(\C)$, we have
natural isomorphisms
\[\varphi_{F,G}\colon \Hom_K(G\otimes_\C F,K)\simeq
\Hom_{\Gr(\C)}(F,D(G))\]
where $\Hom_K(G\otimes_\C F,K)$ is the dual of the graded vector space
$G\otimes_\C F$, and
\[\psi_{F,B}\colon D(B\otimes_\C F)\simeq
\Hom_{\Gr(\C)}(F,D(B))\]
as functors in $\Gr(\C^\op)$.
\item[(d)] $\Gr(\C)$ has enough injectives.
\item[(e)] Assume in addition that $\C$ is a positively graded
Krull-Schmidt category where for each indecomposable object $C$ in
$\C$ the factor $\End_\C(C)/\rad\End_\C(C)$ is finite dimensional over
$K$.

Then the functors of finite length are in $\lfGr(\C)$ (and
respectively $\lfGr(\C^\op)$), and the duality
\[D\colon \lfGr(\C)\to\lfGr(\C^\op)\]
takes functors of finite length to functors of finite length.
\end{enumerate}
\end{prop}
\begin{proof} It is clear that $D$ as defined above gives rise to a
functor from $\Gr(\C)$ to $\Gr(\C^\op)$.

(a) This follows directly from the definitions involved.

(b) Define $\eta_F=\{\eta_F(X)\}_{X\in\C}\colon F\to D^2(F)$ in the
following way. For each object $X$ in $\C$, the set
$D^2(F)(X)_i=\Hom_K(D(F)(X)_{-i},K)$ and
$D(F)(X)_{-i}=\Hom_K(F(X)_{i},K)$. Let $f$ be in $F(X)_{i}$, then
$\eta_F(X)(f)\colon D(F)(X)_{-i}\to K$. Let $g$ be in $D(F)(X)_{-i}$,
then $\eta_F(X)(f)(g)=g(f)$.  Since $F(X)_i$ is a $K$-vector space, as
usual, $\eta_F(X)$ is injective. Hence, $\eta_F(X)$ is injective for
each $X$ in $\C$ and $\eta_F=\{\eta_F(X)\}_{X\in \C}$ is a graded
injective natural transformation.

If $F$ is locally finite, then $(\eta_F(X))_i\colon F(X)_i\to
(D^2(F)(X))_i$ is an isomorphism for all $i$ in $\Z$ and for all $X$
in $\C$. It follows that $\eta_F(X)\colon F(X)\to D^2(F)(X)$ is an
isomorphism and $\eta_F\colon F\to D^2(F)$ is an isomorphism of
functors. The last claim follows directly from this.

(c) Let $G$ be in $\Gr(\C^\op)$. For $M$ in $\C$ we have isomorphisms
\begin{align}
D(G\otimes_\C \Hom_\C(-,M)) & \simeq D(G(M))\notag\\
    & \simeq D(G)(M)\notag\\
    & \simeq \Hom_{\Gr(\C)}(\Hom_\C(-,M),D(G)),\notag
\end{align}
and this give rise to a natural isomorphism $\alpha\colon
D(G\otimes_\C \Hom_\C(-,M))\to \Hom_{\Gr(\C)}(\Hom_\C(-,M),D(G))$.

Now let $F$ be in $\Gr(\C)$. Assume that
\[\amalg_i\Hom_\C(-,M_i)[m_i] \to \amalg_i\Hom_\C(-,N_i)[n_i]\to F\to
0\]
is a graded projective presentation of $F$ in $\Gr(\C)$. This gives
rise to the following commutative diagram with exact rows, where the
horizontal morphisms are induced by the natural isomorphism $\alpha$
described above.
\[\xymatrix@R=10pt{
0\ar[d] & 0\ar[d] \\
{\Hom_{\Gr(\C)}(F,D(G))}\ar[d]\ar[r] & D(G\otimes_\C F) \ar[d]\\
\prod_i {\Hom_{\Gr(\C)}(\Hom_\C(-,N_i)[n_i],D(G))}\ar[d]\ar[r]^-\sim &
\prod_i D(G\otimes_\C \Hom_\C(-,N_i)[n_i])\ar[d]\\
\prod_i {\Hom_{\Gr(\C)}(\Hom_\C(-,M_i)[m_i],D(G))}\ar[r]^-\sim & 
\prod_i D(G\otimes_\C\Hom_\C(-,M_i)[m_i])}\]
Therefore $\varphi=\varphi_{F,G}\colon \Hom_{\Gr(\C)}(F,D(G))\to
D(G\otimes_\C F)$ induced by $\alpha$ is an isomorphism, and the
first claim follows.

Let $B\colon \C^\op\times\C\to \Gr(K)$ be a graded bifunctor, and let
$F$ be as above. Note that $(B\otimes_\C F)(X)=B(X,-)\otimes_\C F$ and
that $\Hom_{\Gr(\C)}(F,D(B))(X) = \Hom_{\Gr(\C)}(F,D(B(X,-)))$. For
any $X$ in $\C$ define
\begin{multline}
\psi_{F,B,X}=\varphi_{F,B(X,-)}\colon
D(B\otimes_\C F)(X)=D(B(X,-)\otimes_\C F)\to \notag\\
\Hom_{\Gr(\C)}(F,D(B(X,-)))=\Hom_{\Gr(\C)}(F,D(B))(X).\notag
\end{multline}
It follows from the above that $\psi_{F,B}$ is a natural isomorphism
of functors.

(d) Let $G$ be a functor in $\Gr(\C)$. Then $D(G)$ in $\Gr(\C^\op)$ is
a factor of a projective functor $\amalg_{j\in
J}\Hom_\C(C_j,-)[n_j]\to D(G)$. The inclusions $G\to D^2(G)$ and
$D^2(G)\to \prod_{j\in J} D((C_j,-))[-n_j]$ induce an inclusion $G\to
\prod_{j\in J}D((C_j,-))[-n_j]$. Hence, if $D(C,-)$ is an injective
functor for all objects $C$ in $\C$, the category $\Gr(\C)$ has enough
injective objects.

Consider an exact sequence of graded functors $0\to F\to G\to H\to 0$
in $\Gr(\C)$. Given the natural isomorphisms
\[\Hom(F,D(\Hom_\C(C,-)))\simeq D(\Hom_\C(C,-)\otimes_\C F)
\simeq D(F(C))\]
and the exactness of the sequence $0\to D(H(C))\to D(G(C))\to
D(F(C))\to 0$ for all objects $C$ in $\C$, we infer that
$\Hom_{\Gr(\C)}(-,D(\Hom_\C(C,-)))$ is an exact functor and
$D(\Hom_\C(C,-)))$ is an injective object in $\Gr(\C)$ for all $C$ in
$\C$.

(e) Under the assumptions in (e) the simple functors are given as
$S_C=\Hom_\C(-,C)/\rad_\C(-,C)$ for some indecomposable object $C$ in
$\C$ up to shift by Lemma \ref{lem:simples}. Since $S_C$ only has
support in $C$ and $S_C(C)=\End_\C(C)/\rad\End_\C(C)$, our assumptions
imply that $S_C$ is locally finite for all indecomposable objects $C$
in $\C$. Hence all functors of finite length are locally finite, that
is, all functors of finite length are in $\lfGr(\C)$.

In (a) we observed that the functor $D$ is a duality from
$\lfGr(\C)$ to $\lfGr(\C^\op)$. It follows from this that if $S$ is a
simple functor, then $D(S)$ is also a simple functor. Consequently, it
follows that if $F$ has finite length, $DF$ also has finite length.
\end{proof}

\subsection{Homological algebra}
Here we discuss some elementary homological facts that we need later
in this series of papers. Among other things we show that a bounded
below flat functor is projective over a positively graded
Krull-Schmidt $K$-category with the radical given by the positive
degrees and that the maximum of the projective dimension of the simple
functors in this setting gives the global dimension of $\Gr(\C)$.

We start with an immediate homological corollary of the duality
considered in the previous subsection.
\begin{cor}
Let $\C$ be a graded $K$-category and consider the functor $D\colon
\Gr(\C)\to \Gr(\C^\op)$ defined above. Then
\[\Hom_K(\Tor^\C_i(G,F),K)\simeq \Ext^i_\C(F,D(G))\]
for all $F$ in $\Gr(\C)$ and $G$ in $\Gr(\C^\op)$, and all $i\geq 0$.
\end{cor}

In the following let $\C$ be a positively graded Krull-Schmidt
$K$-category with $\rad_\C(-,-)=\amalg_{i\geq 1}\Hom_\C(-,-)_i$. In
this setting we discuss bounded below functors. To this end we need to
observe the following. As we saw in Lemma \ref{lem:simples} all simple
functors in $\Gr(\C)$ are given as
$S_C=\Hom_\C(-,C)/\rad_\C(-,C)=\Hom_\C(-,C)_0$ for some indecomposable
object $C$ in $\C$. Then we have the following.

\begin{prop}\label{prop:homological}
Let $\C$ be a positively graded Krull-Schmidt $K$-category with
$\rad_\C(-,-)=\amalg_{i\geq 1}\Hom_\C(-,-)_i$.  Let $F$ be a
bounded below functor in $\Gr(\C)$.
\begin{enumerate}
\item[(a)] If $S\otimes_\C F=(0)$ for all simple functors $S$ in
$\Gr(\C^\op)$, then $F$ is zero.
\item[(b)] If $\Tor^\C_1(S,F)=(0)$ for all simple functors $S$ in
$\Gr(\C^\op)$, then $F$ is a projective object in $\Gr(\C)$. In
particular, a bounded below flat functor is a projective functor.
\item[(c)] Assume that each simple object $S$ in $\Gr(\C^\op)$ has
projective dimension at most $n$. Then $F$ has projective dimension at
most $n$.
\end{enumerate}
\end{prop}
\begin{proof}
(a) Let $F$ be a bounded below functor in $\Gr(\C)$. For any
simple functor $S=\Hom_\C(C,-)/\rad_\C(C,-)$ we have that
\[(0)=S\otimes_\C F=\Hom_\C(C,-)/\rad_\C(C,-)\otimes_\C F\simeq
F(C)/(\rad_\C F)(C)\]
for all indecomposable $C$ in $\C$. It follows that $F/\rad_\C
F=(0)$. By Nakayama's Lemma we infer that $F=(0)$.

(b) Let $F$ be a bounded below functor in $\Gr(\C)$, and let
\[0\to \Omega_\C^1(F)\to P \to F\to 0\]
be a projective cover of $F$. Since $\C$ is positively graded, it follows
that $\Omega_\C^1(F)$ is also bounded below. The above exact sequence
gives rise to the exact sequence
\[0\to \Tor^\C_1(S,F)\to S\otimes_\C \Omega^1_\C(F)\to S\otimes_\C
P\to S\otimes_\C F\to 0\]
by Lemma \ref{lem:posgradprojcover} (b).
For $S=\Hom_\C(C,-)/\rad_\C(C,-)$ with $C$ indecomposable in $\C$, we
have that
\[S\otimes_\C P\simeq P(C)/\rad_\C P(C)\simeq F(C)/\rad_\C F(C)\simeq
S\otimes_\C F.\]
Therefore $(0)=\Tor^\C_1(S,F)\simeq S\otimes_\C \Omega^1_\C(F)$. By
(a) we conclude that $\Omega^1_\C(F)=(0)$ and that $F$ is projective.

(c) By assumption $\Tor^\C_{n+1}(S,-)=(0)$ for all simple functors $S$
in $\Gr(\C^\op)$. By dimension shift $\Tor^\C_1(S,\Omega^n_\C(F))=(0)$
for any bounded below functor $F$ in $\Gr(\C)$. Since $\Omega^n_\C(F)$
is bounded below, we infer from (b) that $\Omega^n_\C(F)$ is
projective and that the projective dimension of $F$ is at most $n$.
\end{proof}
Using this we show that the maximum of the projective dimensions of
the simple functors is the global dimension of the category $\Gr(\C)$.

\begin{thm}\label{thm:globaldim}
Let $\C$ be a positively graded Krull-Schmidt $K$-category with
$\rad_\C(-,-)=\amalg_{i\geq 1}\Hom_\C(-,-)_i$.
\begin{enumerate}
\item[(a)] Assume that all simple functors $S$ in $\Gr(\C^\op)$
have projective dimension at most $n$.  Then $\Gr(\C)$ has global
dimension at most $n$.
\item[(b)] The global dimension of $\Gr(\C)$ is finite if and only if
  the global dimension of $\Gr(\C^\op)$ is finite. When one of the
  global dimensions is finite, then they are equal.
\end{enumerate}
\end{thm}
\begin{proof}
(a) We proved in Proposition \ref{prop:homological} that any bounded below
functor $F$ in $\Gr(\C)$ has projective dimension at most $n$. Let $F$
be any graded functor. Define $G_i=F_{\geq -i}$ for all $i\geq 0$.  We
obtain a directed system $G_0\hookrightarrow G_1\hookrightarrow
G_2\hookrightarrow \cdots$, and $\varinjlim G_i = F$.

We have exact sequences of functors
\[0\to G_i\to G_{i+1}\to F_{-i-1}\to 0,\]
so that the projective objects of a projective resolution of $G_{i+1}$
can be chosen as the direct sum of the projective objects occurring in
the projective resolution of $G_i$ and $F_{-i-1}$ and the morphism
between the resolutions for $G_i$ and $G_{i+1}$ would just be the
natural inclusion. Let $P^i$ and $Q^i_j$ be the projective objects
occurring at stage $i$ in the projective resolution of $G_0$ and of
$F_j$, respectively. If we take the direct limit of this system, we
obtain as the projective occurring at stage $i$ to be $P^i\amalg
(\amalg_{j<0} Q^i_j)$. Since $G_0$ and all $F_j$ are bounded below,
the projectives occurring at stage $n+1$ are zero, so that the $n$-th
syzygy in the induced resolution of $F$ is $P^n\amalg (\amalg_{j<0}
Q^n_j)$. Hence $F$ has projective dimension at most $n$.

(b) The follows directly from (a) and Proposition
\ref{prop:homological} (c) for $\Gr(\C)$.
\end{proof}

\section{Koszul categories}\label{section2}
This section is devoted to defining Koszul categories and showing some
elementary properties of them. In particular it is shown that a Koszul
category is generated in degrees $0$ and $1$, and that there is a
naturally associated Koszul dual.

First we define our Koszul categories.
\begin{defin}
Let $\C$ be a positively graded Krull-Schmidt locally finite
$K$-category.
\begin{enumerate}
\item[(a)] A functor $F$ in $\Gr(\C)$ is said to be \emph{linear} if
$F$ has a finitely generated projective graded resolution of the form
\[\cdots\to \Hom_\C(-,C_2)[-2]\to \Hom_\C(-,C_1)[-1]\to
\Hom_\C(-,C_0)\to F\to 0\] with $C_i$ in $\C$ for all $i\geq 0$ and
all morphisms have degree zero.
\item[(b)] The category $\C$ is said to be \emph{Koszul} if all simple
functors $S$ in $\Gr(\C)$ are linear.
\end{enumerate}
\end{defin}
Our notion of a Koszul category include the classical Koszul algebras
defined in \cite{P}. However the generalization in \cite{B} and the
further generalization in \cite{CS} of Koszul algebras are not
covered. 

Let $\C$ be a positively graded $K$-category. Recall that the category
$\C$ is said to be \emph{generated in degrees $0$ and $1$} if the
ideal $J=\amalg_{i\geq 1}\Hom_\C(-,-)_i\subseteq \Hom_\C(-,-)$
satisfies $J^r=\amalg_{i\geq r}\Hom_\C(-,-)_i$ for all $r\geq 1$.

Next we show that Koszul categories share the same property as Koszul
algebras being generated in degrees $0$ and $1$. 
\begin{lem}\label{lem:genin0and1}
Let $\C$ be a Koszul category. Then the following assertions are
true.
\begin{enumerate}
\item[(a)] $\rad_\C(-,-)=\Hom_\C(-,-)_{\geq 1}$.
\item[(b)] The category $\C$ is generated in degrees $0$ and $1$.
\end{enumerate}
\end{lem}
\begin{proof}
  (a) Since $\C$ is Koszul, the simple functors in $\Gr(\C)$ are
  finitely presented. By Lemma \ref{lem:simples} the simple functors
  are given as $S_C=\Hom_\C(-,C)/\rad_\C(-,C)$ for some indecomposable
  object $C$ in $\C$ up to shift. The start of the graded projective
  resolution of $S_C$ is then of the form
  \[0\to \rad_\C(-,C)\to \Hom_\C(-,C)\to S_C\to 0.\] The simple
  functors $S_C$ are only concentrated in one degree, so
  $\Hom_\C(-,C)_{\geq 1}$ is contained in $\rad_\C(-,C)$. The next
  projective in the graded projective resolution of $S_C$ being of the
  form $\Hom_\C(-,C_1)[-1]$, implies that $\rad_\C(-,C)$ is contained
  in $\Hom_\C(-,C)_{\geq 1}$. Hence $\rad_\C(-,C)=\Hom_\C(-,C)_{\geq
    1}$ for all indecomposable objects $C$ in $\C$. The claim follows
  from this.

(b) The epimorphism $\Hom_\C(-,C_1)[-1]\to \Hom_\C(-,C)_{\geq 1}$ for
any indecomposable object $C$ in $\C$ found in (a), implies that
$\Hom_\C(X,C)_n=\Hom_\C(C_1,C)_1 \Hom_\C(X,C_1)_{n-1}$ for all $n\geq
2$ and all objects $X$ in $\C$. The claim follows by induction from
this.
\end{proof}

Next we show that Koszul categories have many of the same properties
as those of Koszul algebras. Let $\C$ be a Koszul $K$-category, and
let $\S(\C)$ be the full additive subcategory generated by the simple 
objects in $\Gr(\C)$, which might be viewed as a subcategory of
$\Mod(\C)$. Recall that the $\Ext$-category (see Example
\ref{exam:extcat} in Section \ref{section1}) $E(\S(\C))$ has the same
objects as $\S(\C)$ and the morphisms are given by 
\[\Hom_{E(\S(\C))}(A,B)=\oplus_{i\geq 0}\Ext^i_{\Mod(\C)}(A,B).\]
Since $E(\S(\C))$ consists of objects in $\Mod(\C)$, for every object
$F$ in $\Mod(\C)$ we can consider $\Ext^i_{\Mod(\C)}(F,A)$ for any $i$
and for any object $A$ in $E(\S(\C))$. This gives rise to an 
analogue of the usual Koszul duality functor $\phi\colon \Gr(\C)\to
\Gr(E(\S(\C)))$ given by 
\[\phi(F)=\Ext^*_{\Mod(\C)}(F,-)=\oplus_{i\geq
  0}\Ext^i_{\Mod(\C)}(F,-).\] 
Indeed note that this is a contravariant functor. For related algebra
results see \cite[Proposition 5.5 (b)]{GM1} for (b), \cite[Theorem
5.2]{GM2} for (e), \cite[Theorem 2.10.2]{BGS} or \cite[Theorem
6.1]{GM1} for (f), \cite[Theorem 2.10.2]{BGS} or \cite[Theorem
2.3]{GM2} for (g).
\begin{thm}\label{thm:koszul}
Let $\C$ be a Koszul $K$-category. Then the following assertions are
true.
\begin{enumerate}
\item[(a)] If $F$ is linear, then $\Omega^1_{\Gr(\C)}(F)[1]$ is
  linear.
\item[(b)] If $F$ is linear, then $\rad_\C F[1]$ is linear. In
  particular $\rad_\C F$ is finitely generated.
\item[(c)] Let $0\to F_1\to F_2\to F_3\to 0$ be an exact sequence in
$\Gr(\C)_0$ of functors generated in degree zero. Then $\rad_\C F_2\cap
F_1=\rad_\C F_1$.
\item[(d)] If $0\to F_1\to F_2\to F_3\to 0$ is exact in $\Gr(\C)_0$
with $F_1$ and $F_2$ linear, then $F_3$ is linear.
\item[(e)] Let $\phi\colon \Gr(\C)\to \Gr(E(\S(\C))^\op)$ be given by
  $\phi(F)=\Ext^*_{\Mod(\C)}(F,-)$, where $\S(\C)$ is the full
  additive subcategory generated by the simple functors in
  $\Gr(\C)$. The functor $\phi$ restricts to a functor from the linear
  functors in $\Gr(\C)$ to the linear functors in
  $\Gr(E(\S(\C))^\op)$.
\item[(f)] The graded $K$-category $\C'=E(\S(\C))^\op$ is Koszul.
\item[(g)] The graded $K$-categories $\C$ and $E(\S(\C'))^\op$
are equivalent.
\end{enumerate}
\end{thm}
\begin{proof}
The claim in (a) is clear from the definition of linear functors. The
statement in (b) follows in a similar way as for algebras using that
if $0\to F_1\to F_2\to F_3\to 0$ is an exact sequence in $\Gr(\C)$
with degree zero homomorphisms and $F_1$ and $F_2$ linear, then $F_3$
is also linear.

(c) Let $0\to F_1\to F_2\to F_3\to 0$ be an exact sequence in
$\Gr(\C)_0$ of functors generated in degree zero. This exact sequence
induces the following commutative exact diagram
\[\xymatrix{
 & 0\ar[d] & 0\ar[d] & 0\ar[d] & \\
0\ar[r] & F_1\cap \rad_\C F_2\ar[r]\ar[d] & \rad_\C F_2\ar[r]\ar[d] & \rad_\C
F_3 \ar[r]\ar[d] & 0\\
0\ar[r] & F_1 \ar[r] & F_2\ar[r] & F_3\ar[r] & 0}\]
Since $F_2$ is generated in degree zero, $\rad_\C F_2=(F_2)_{\geq 1}$ and
$\rad_\C F_2(X)_0=(0)$ for all $X$ in $\C$. Consequently $(F_1\cap \rad_\C
F_2)(X)_0=(0)$ for all $X$. Similarly, $\rad_\C F_1=(F_1)_{\geq 1}$ and
$\rad_\C F_1(X)_i=F_1(X)_i$ for $i\geq 1$ and $\rad_\C F_1(X)_0=(0)$ for all
$X$. It follows that $F_1\cap \rad_\C F_2\subseteq (F_1)_{\geq 1}=\rad_\C
F_1$.  But in general $\rad_\C F_1\subseteq F_1\cap \rad_\C F_2$, hence
$\rad_\C F_1=F_1\cap \rad_\C F_2$.

(d) Let $0\to F_1\to F_2\to F_3\to 0$ be exact in $\Gr(\C)_0$ with
$F_1$ and $F_2$
linear. By (b) we infer that $0\to F_1/\rad_\C F_1\to F_2/\rad_\C F_2\to
F_3/\rad_\C F_3\to 0$ is exact, and therefore $0\to \Omega(F_1)\to
\Omega(F_2) \to \Omega(F_3)\to 0$ is exact. Hence $\Omega(F_3)$ is
generated in degree $0$. By induction we infer that $F_3$ is linear.

(e) Given a linear functor $F$ in $\Gr(\C)$ there are exact sequences
\[0\to \Omega^i(F)\to \Omega^i(F/\rad_\C F)\to \Omega^{i-1}\rad_\C F\to 0\]
for all $i\geq 1$. Similar arguments as for modules then show that
there is an exact sequence
\[0\to \Ext^*_{\Mod(\C)}(\rad_\C F,-)[-1]\to\Ext^*_{\Mod(\C)}(F/\rad_\C F,-)\to
\Ext^*_{\Mod(\C)}(F,-)\to 0,\]
where $F/\rad_\C F$ is a finite direct sum of simple functors.  Since
$\rad_\C F$ is shift of a linear functor, it follows by induction that
$\phi(F)$ is a linear functor in $\Gr(E(\S(\C))^\op)$.

(f) The category $\C'=E(\S(\C))^\op$ is a positively graded
Krull-Schmidt category, since $\C'$ is an additive $K$-category
generated by the simple objects in $\Gr(\C)$ and each simple object
has a graded local endomorphism ring. It is locally finite, since all
functors in $\S(\C)$ are linear. The simple functors in $\Gr(\C')$ are
given by $\widehat{S}_C=\Ext^*(S_C,-)/\rad\Ext^*(S_C,-)=\Hom(S_C,-)$
for some indecomposable object $C$ in $\C$. It is easy to see that
$\phi(\Hom_\C(-,C))\simeq \widehat{S}_C$ for all indecomposable
objects $C$ in $\C$. By (e) we infer that $\C'$ is a Koszul
$K$-category.

(g) Let $\C'=E(\S(\C))^\op$. Define $H\colon \C\to E(\S(\C'))^\op$
first on objects by letting for $C=\amalg_{i=1}^n C_i$ with $C_i$
indecomposable in $\C$ for $i=1,2,\ldots,n$
\[H(C) = \amalg_{i=1}^n\widehat{S}_{C_i}.\]
A morphism $f\colon D\to C$ between two indecomposable objects in
degree $i$, that is, $f$ is in $\rad_\C^i(D,C)/\rad_\C^{i+1}(D,C)$, gives
rise to a morphism of functors $\Hom_\C(-,D)\to
\rad_\C^i(-,C)/\rad_\C^{i+1}(-,C)$. Applying the functor $\phi$ to this
morphism induces a morphism $\Ext^*(\rad_\C^i(-,C)/\rad_\C^{i+1}(-,C),-)\to
\widehat{S}_D$. This we can interpret as an element $H(f)$ in
$\Ext^i(\widehat{S}_C,\widehat{S}_D)$. This is linearly extended to
any morphism in $\C$. This defines $H$ on morphisms in $\C$, keeping in
mind that our target category is $E(\S(\C'))^\op$. It is clear that
$H$ takes an identity morphism to an identity morphism.

Using the proof of (a) and (e), the simple
functors $\widehat{S}_C=\Hom(S_C,-)$ in $\Gr(\C')$ have a minimal
projective resolution given by
\begin{multline}
\cdots\to \Ext^*(\rad_\C^2(-,C)/\rad_\C^3(-,C),-)[-2] \to\\
\Ext^*(\rad_\C(-,C)/\rad_\C^2(-,C),-)[-1]\to\\
\Ext^*(\Hom(-,C)/\rad_\C(-,C),-)\to \widehat{S}_C\to 0\notag
\end{multline}
We infer from this that
\begin{align}
\Ext^i(\widehat{S}_C,\widehat{S}_D) & \simeq
\Hom(\Ext^*(\rad_\C^i(-,C)/\rad_\C^{i+1}(-,C),-)[-i],\widehat{S}_D)\notag\\
& \simeq
\Hom(\Hom(\rad_\C^i(-,C)/\rad_\C^{i+1}(-,C),-)[-i],\widehat{S}_D)\notag\\ 
& \simeq \widehat{S}_D(\rad_\C^i(-,C)/\rad_\C^{i+1}(-,C)[-i])\notag\\
& = \Hom(S_D,\rad_\C^i(-,C)/\rad_\C^{i+1}(-,C)[-i])\notag\\
& = \Hom(\Hom_\C(-,D)/\rad_\C
(-,D),\rad_\C^i(-,C)/\rad_\C^{i+1}(-,C)[-i])\notag\\ 
& \simeq \Hom(\Hom_\C(-,D),\rad_\C^i(-,C)/\rad_\C^{i+1}(-,C)[-i])\notag\\
& \simeq \rad_\C^i(D,C)/\rad_\C^{i+1}(D,C)[-i]\notag\\
& = \Hom_\C(D,C)_i\notag
\end{align}
Tracing through all these isomorphisms one can show that this is the
morphism defining the functor $H$ on morphisms in $\C$. Hence $H$
is full and faithful.

Finally we need to show that $H$ commutes with composition of
maps. Let $f$ be in $\rad_\C^i(D,C)/\rad_\C^{i+1}(D,C)$, and let $g$ be in
$\rad_\C^i(C,D')/\rad_\C^{i+1}(C,D')$. We want to show that $H(gf)$ is the
Yoneda product of $H(f)$ and $H(g)$. In order to do so, we need to
lift the morphism $\Ext^*(\rad_\C^j(-,C')/\rad_\C^{j+1}(-,C'),-)\to
\Ext^*((-,C)/\rad_\C(-,C),-)$ induced by $g$, through a chain map to a
morphism
\[\Ext^*(\rad_\C^{i+j}(-,C')/\rad_\C^{i+j+1}(-,C'),-)\to
\Ext^*(\rad_\C^i(-,C)/\rad_\C^{i+1}(-,C),-).\]
It is easily seen that this morphism also is induced by $g$, so that
it follows that the Yoneda product of $H(f)$ and $H(g)$ is given by
$H(gf)$. This shows that $H\colon \C\to E(\S(\C'))^\op$ is an
equivalence of graded $K$-categories.
\end{proof}

In view of this result, for a Koszul $K$-category $\C$ we call the
category $E(\S(\C))$ the \emph{Koszul dual of $\C$}.

The prime example and the main application of the theory is related to
the category of all additive functors from $(\mod\L)^\op$ to vector
spaces for a finite dimensional $K$-algebra $\L$. This application
will be presented in a forthcoming paper. As a further motivation for
this paper and in particular for the next section we include a brief
discussion of this application.

Let $\L$ be a finite dimensional algebra over a field $K$. Denote by
$\mod\L$ the category of finitely generated left $\L$-modules, and by
$\Mod(\mod\L)$ the category of all additive functors $F\colon
(\mod\L)^\op\to \Mod K$.
As we saw in Example \ref{exam:assgraded} we can consider the
associated graded category $\A_\gr(\mod\L)$ of $\Mod(\mod\L)$. It has
the same objects as $\mod\L$, while the morphisms are given by
\[\Hom_{\A_\gr(\mod\L)}(X,Y)=\amalg_{i\geq
  0}\rad_\C^i(X,Y)/\rad_\C^{i+1}(X,Y).\]
The simple functors in $\Mod(\mod\L)$ are of the form
$S_C=(-,C)/\rad_\C(-,C)$ for some indecomposable $C$ in $\mod\L$. They
have a minimal projective resolution given by
\[0\to \Hom_\L(-,\tau C)\to \Hom_\L(-,E)\to \Hom_\L(-,C)\to S_C\to 0\]
when $C$ is non-projective with corresponding almost split sequence
$0\to \tau C\to E\to C\to 0$, and given by
\[0\to \Hom_\L(-,\rrad P)\to \Hom_\L(-,P)\to S_P\to 0\]
when $P$ is an indecomposable projective $\L$-module with radical
$\rrad P$.

Recalling the construction in subsection \ref{subsec:ideals} we have a
functor $G\colon \Mod(\mod\L)\to \Gr(\A_\gr(\mod\L))$ on objects and
morphisms given for a functor $F$ in $\Mod(\mod\L)$ by
\[G(F)=\amalg_{i\geq 0} \rad^iF/\rad^{i+1}F.\]
The simple functors in $\Gr(\A_\gr(\mod\L))$ are all of the form
$G(S_C)$ for some $C$ in $\mod\L$, and they have a minimal projective
resolutions given by
\begin{multline}
0\to \amalg_{i\geq 0}\rad^i(-,\tau C)/\rad^{i+1}(-,\tau C)[-2]\to\\
\amalg_{i\geq 0}\rad^i(-,E)/\rad^{i+1}(-,E)[-1]\to \\
\amalg_{i\geq 0}\rad^i(-,C)/\rad^{i+1}(-,C)\to
G(S_C)\to 0\notag
\end{multline}
in case $C$ is non-projective, and given by
\begin{multline}
0\to \amalg_{i\geq 0}\rad^i(-,\rrad P)/\rad^{i+1}(-,\rrad
P)[-1]\to \\
\amalg_{i\geq 0}\rad^i(-,P)/\rad^{i+1}(-,P)\to
G(S_P)\to 0\notag
\end{multline}
when $P$ is projective by \cite{IT}. This shows that
$\Gr(\A_\gr(\mod\L))$ is a Koszul $K$-category. This illustrate the
content of the next section. There we define weakly Koszul
$K$-categories $\C$, which $\Mod(\mod\L)$ is an example of. We
study the relationship with $\A_\gr(\C)$ in general, and show that
$\A_\gr(\C)$ always is a Koszul $K$-category.

The application is related to the representation theory of $\L$ and in
particular to the Auslander-Reiten theory. The indecomposable modules
in $\mod\L$ is a disjoint union $\cup_{\sigma\in \Sigma}C_\sigma$,
where each $C_\sigma$ is an component of the Auslander-Reiten quiver
of $\L$. Furthermore, for $X$ and $Y$ in $\add C_\sigma$ and $\add
C_{\sigma'}$ respectively with $\sigma\neq \sigma'$, we have
\[\Hom_{\A_\gr(\mod\L)}(X,Y)=\amalg_{i\geq 0}\rad^i(X,Y)/\rad^{i+1}(X,Y)=(0)\]
since any morphism from $X$ to $Y$ is in the infinite radical.  Hence
$\A_\gr(\mod\L)$ is a disjoint union
$\cup_{\sigma\in\Sigma}\A_\gr(\add\C_\sigma)$ of categories.  As a
consequence the category $\Gr(\A_\gr(\mod\L))$ is the product
$\prod_{\sigma\in\Sigma} \Gr(\A_\gr(\add\C_\sigma))$, which all are
Koszul $K$-categories. We show in a forthcoming paper that properties
of these categories reflect properties of the component $\C_\sigma$.

\section{Weakly-Koszul categories}\label{section3}
Here we introduce the non-graded analogue of Koszul categories similar
as was done for algebras in \cite{MVZ}. The primary example for us of
a weakly Koszul category is the category of additive functors from
$(\mod\L)^\op$ to vector spaces for a finite dimensional $K$-algebra
$\L$.

Let $\C$ be an additive Krull-Schmidt $K$-category. This in particular
implies that $\End_\C(C)$ is semiperfect for all objects in $C$ in
$\C$. Consequently the category of finitely presented functors
$\mod\C$ in $\Mod(\C)$ has minimal projective presentations (compare
Lemma \ref{lem:minpresent}). Also, similarly as in Lemma
\ref{lem:simples} this gives rise to one-to-one correspondence between
the indecomposable objects in $\C$ and the simple objects in $\Mod(\C)$,
where an indecomposable object $C$ in $\C$ gives rise to the simple
object $S_C=\Hom_\C(-,C)/\rad_\C(-,C)$ in $\Mod(\C)$.

We need a further finiteness condition to define weakly Koszul
$K$-categories. A $K$-category $\C$ is called \emph{locally radical
finite} if $\rad_\C^i(A,B)/\rad_\C^{i+1}(A,B)$ is finite dimensional over
$K$ for all pairs of objects $(A,B)$ in $\C$ and all $i\geq 0$.
Throughout this section let $\C$ denote an additive Krull-Schmidt
locally radical finite $K$-category. First we define a weakly Koszul
$K$-category.
\begin{defin}
\begin{enumerate}
\item[(i)] A functor $F$ in $\Mod(\C)$ is \emph{weakly Koszul} if $F$
has a projective resolution
\[\cdots\to P_n\to P_{n-1}\to\cdots\to P_1\to P_0\to F\to 0\]
where $P_i$ is a finitely generated projective object in $\Mod(\C)$ for
all $i\geq 0$ and $\rad_\C^{i+1}(P_j)\cap
\Omega^{j+1}(F)=\rad_\C^i(\Omega^{j+1}(F))$ for all $j\geq 0$ and
$i\geq 1$.
\item[(ii)] A category $\C$ is \emph{weakly Koszul}, if $\C$ is an
additive Krull-Schmidt locally radical finite $K$-category and every
simple functor in $\Mod(\C)$ is weakly Koszul.
\end{enumerate}
\end{defin}
Note that if a functor $F$ is weakly Koszul, then $\Omega^i(F)$ is
weakly Koszul for all $i\geq 0$.

For algebras the associated graded algebra, with respect to the 
Jacobson radical, of a weakly Koszul algebra is Koszul. We have the
same for our weakly Koszul categories, as we show next. First recall
from Example \ref{exam:assgraded} and from subsection
\ref{subsec:ideals} how we from an additive $K$-category $\C$
constructed the associated positively graded $K$-category $\A_\gr(\C)$
and a functor $G\colon\Mod(\C)\to \Gr(\A_\gr(\C))$ on objects and
morphisms, where for $F$ in $\Mod(\C)$ the functor $G$ is given by
$G(F)=\amalg_{i\geq 0}\rad_\C^iF/\rad_\C^{i+1}F$.
\begin{prop}\label{prop:weaklytoKoszul}
Suppose that $\C$ is an additive Krull-Schmidt locally radical finite
$K$-category.  Let $F$ in $\Mod(\C)$ be weakly Koszul. Then $G(F)$ in
$\Gr(A_\gr(\C))$ is Koszul. Moreover, if $\C$ is weakly Koszul, then
$\A_\gr(\C)$ is Koszul.
\end{prop}
\begin{proof}
  The category $\A_\gr(\C)$ is clearly positively graded and locally
  finite, and it is Krull-Schmidt as the associated graded ring of a
  local ring is graded local.

Let \[\cdots\to\Hom_\C(-,C_2)\to\Hom_\C(-,C_1)\to
\Hom_\C(-,C_0)\extto{d} F\to 0\]
be a minimal projective resolution of $F$. By Lemma \ref{lem:idealepi}
(b) the map
\[d|_{\rad_\C^i \Hom_\C(-,C_0)}\colon \rad_\C^i \Hom_\C(-,C_0)\to
\rad_\C^i F\]
is onto for all $i\geq 0$. The kernel of $d|_{\rad_\C^i
  \Hom_\C(-,C_0)}$ is $(\rad_\C^i \Hom_\C(-,C_0))\cap \Omega(F)$, which
is equal to $\rad_\C^{i-1}\Omega(F)$, since $F$ is weakly Koszul. This
gives rise to the exact sequences
\[0\to\rad_\C^i\Omega(F)\to \rad_\C^{i+1}\Hom_\C(-,C_0)\to \rad_\C^{i+1}F\to
0\]
for all $i\geq 0$. Consequently there are exact sequences
\begin{multline}
0\to\rad_\C^{i-1}\Omega(F)/\rad_\C^i\Omega(F)\to
\rad^i_\C(-,C_0)/\rad^{i+1}_\C(-,C_0) \to\\
\rad_\C^iF/\rad_\C^{i+1}F\to 0\notag
\end{multline}
for all $i\geq 1$. Combining all these sequences we obtain the exact
sequence
\[0\to G(\Omega(F))[1]\to G(\Hom_\C(-,C_0))\to G(F)\to 0\]
of functors. By induction we have an exact sequence
\[\cdots\to G(\Hom_\C(-,C_2))[2]\to G(\Hom_\C(-,C_0))[1]\to G(F)\to
0\]
We infer that $G(F)$ is a linear $\A_\gr(\C)$-module.

Suppose that $\C$ is weakly Koszul. Since any simple
$\A_\gr(\C)$-module is of the form $G(S_C)$, it follows that
$\A_\gr(\C)$ is Koszul.
\end{proof}

In the following result and its corollary we show that weakly Koszul
modules are closed under cokernels of monomorphisms and that the
radical of a weakly Koszul modules again is weakly Koszul.

\begin{prop}\label{prop:closedundermono}
Suppose that $\C$ is weakly Koszul. Let $0\to F_1\to F_2\to F_3\to 0$
is an exact sequence in $\mod\C$. Assume that $(\rad_\C^k
F_2)\cap F_1=\rad_\C^k F_1$ for all $k\geq 0$ and that $F_1$ and $F_2$
are weakly Koszul. Then $F_3$ is weakly Koszul.
\end{prop}
\begin{proof}
By assumption the sequence
\[0\to F_1/\rad_\C F_1\to F_2/\rad_\C F_2\to F_3/\rad_\C F_3\to 0\]
is exact. Then we have an exact commutative diagram
\[\xymatrix{%
        & 0\ar[d] & 0\ar[d] & 0\ar[d] & \\
0\ar[r] & \Omega(F_1)\ar[r]\ar[d] & \Omega(F_2)\ar[r]\ar[d] &
\Omega(F_3) \ar[r]\ar[d] & 0 \\
0\ar[r] & \Hom_\C(-,C_1)\ar[r]\ar[d] & \Hom_\C(-,C_2)\ar[r]\ar[d] &
\Hom_\C(-,C_3)\ar[r]\ar[d] & 0\\
0\ar[r] & F_1\ar[r]\ar[d] & F_2\ar[r]\ar[d] & F_3 \ar[r]\ar[d] & 0\\
        & 0                       & 0                       & 0 & }\]
with $\Hom_\C(-,C_i)/\rad_\C(-,C_i)\simeq F_i/\rad F_i$ for
$i=1,2,3$. Using the assumptions we have
\begin{align}
(\rad_\C^i\Omega(F_2))\cap \Omega(F_1)
& = (\rad^{i+1}_\C(-,C_2)) \cap \Omega(F_2)\cap \Omega(F_1)\notag\\
& = \rad^{i+1}_\C(-,C_2)\cap \Hom_\C(-,C_1)\cap \Omega(F_1)\notag\\
& =\rad^{i+1}_\C(-,C_1)\cap \Omega(F_1)=\rad_\C^i\Omega(F_1)\notag
\end{align}
for all $i\geq 1$. Hence we have a commutative diagram
\[\xymatrix@C=10pt{%
        & 0\ar[d] & 0\ar[d] & 0\ar[d] & \\
0\ar[r] & \rad_\C^i\Omega(F_1)\ar[r]\ar[d] & \rad_\C^i\Omega(F_2)\ar[r]\ar[d] &
\rad_\C^i\Omega(F_3) \ar[r]\ar[d] & 0 \\
0\ar[r] & \rad^{i+1}_\C(-,C_1)\ar[r]\ar[d] &
\rad^{i+1}_\C(-,C_2)\ar[r]\ar[d]  &
\rad^{i+1}_\C(-,C_3)\ar[r]\ar[d] & 0\\
0\ar[r] & \rad_\C^{i+1}F_1\ar[r]\ar[d] & \rad_\C^{i+1} F_2\ar[r]\ar[d] &
\rad_\C^{i+1}F_3 \ar[r]\ar[d] & 0\\
        & 0                       & 0                       & 0 & }\]
Since $\rad_\C^i\Omega(F_3)\subseteq \rad_\C^{i+1}\Hom_\C(-,C_3)\cap
\Omega(F_3)$, the last column is a complex and the remaining columns
and all the rows are exact. It follows by the Snake Lemma, that the
rightmost column also is exact. Therefore
\[\rad_\C^i\Omega(F_3)=\rad^{i+1}_\C(-,C_3)\cap \Omega(F_3).\]
By induction $\Omega(F_3)$ is weakly Koszul. Then $F_3$ is weakly
Koszul, and this completes the proof.
\end{proof}

\begin{cor}\label{cor:radweaklykoszul}
Let $\C$ be a weakly Koszul $K$-category. If $F$ is weakly Koszul,
then $\rad_\C F$ is weakly Koszul. In particular $\rad_\C F$ is finitely
generated.
\end{cor}
\begin{proof}
The sequence, $0\to \Omega(F)\to \Hom_\C(-,C)\to F\to 0$ induces an
exact sequence $0\to \Omega(F)\to \rad_\C(-,C)\to \rad_\C F\to 0$,
satisfying the conditions of Proposition
\ref{prop:closedundermono}. Therefore $\rad_\C F$ is weakly Koszul.
\end{proof}

To further illuminate the relationship between linear objects and
weakly Koszul objects for Koszul categories, we show that a weakly
Koszul object generated in degree zero in a Koszul category is a
linear object. This is true even more general. Recall that an object
$F$ in $\Mod(\C)$ is \emph{quasi-Koszul} if $F$ has a finitely
generated projective resolution 
\[\cdots \to \Hom_\C(-,C_i)\to \cdots \to \Hom_\C(-,C_0)\to F\to 0\]
and that
$\rad_\C\Omega^i(F)=\rad^2_\C(-,C_{i-1})\cap \Omega^i(F)$ for all
$i\geq 0$. 

\begin{lem}
Let $\C$ be a Koszul algebra, and let $F$ be in $\Gr(\C)$.
Assume that $F$ is weakly Koszul (or weaker, quasi-Koszul)
and generated in degree zero. Then $F$ is linear.
\end{lem}
\begin{proof}
Since $F$ is weakly Koszul (or quasi-Koszul), there is an exact
sequence
\[0\to \Omega(F)\to \Hom_\C(-,C)\to F\to 0\]
with $\rad_\C\Omega(F)=\rad^2_\C(-,C)\cap \Omega(F)$. By induction it 
is enough to prove that $\Omega(F)$ is generated in degree $1$.

By the above observation we have the following exact commutative
diagram
\[\xymatrix@R=15pt@C=15pt{
0\ar[r] & \Omega(F)\ar[r]\ar[d] & \rad_\C(-,C)
\ar[r]\ar[d] & \rad_\C F\ar[r]\ar[d] & 0\\
0\ar[r] & \Omega(F)/\rad_\C\Omega(F)\ar[r]\ar[d] &
\rad_\C(-,C)/\rad_\C^2(-,C)\ar[r]\ar[d] & \rad_\C F/\rad_\C^2
  F\ar[r]\ar[d] & 0\\
& 0 & 0 & 0 & }\]
Hence $\Omega(F)/\rad_\C\Omega(F)$ is generated in degree $1$. It
follows that $\Omega(F)$ has a projective cover generated in degree
$1$. By induction $F$ is linear.
\end{proof}

The next result indicates that for a weakly Koszul $K$-category $\C$
there is in addition to the naturally associated Koszul category
$\A_\gr(\C)$, a second associated category and possibly a Koszul
category, namely $E(\S(\C))$. We shall later see that these categories
are Koszul dual of each other.
\begin{prop}
Let $\C$ be a weakly Koszul $K$-category. Then the functor $\phi\colon
\Mod(\C)\to \Gr(E(\S(\C))^\op)$ given by $\phi(F)=\Ext^*_{\Mod(\C)}(F,-)$
restricts to a functor from the category weakly Koszul modules
$\wK(\C)$ to the category of linear functors in $\Gr(E(\S(\C))^\op)$.
\end{prop}
\begin{proof}
It is clear that $\phi$ gives rise to a functor from $\Mod(\C)$ to
$\Gr(E(\S(\C))^\op)$. Next we show that $\Ext^*_{\Mod(\C)}(F,-)$ is a
linear $E(\S(\C))^\op$-module when $F$ is a weakly Koszul $\C$-module.

Let $F$ be weakly Koszul. Let $\Hom_\C(-,C)\to F$ be a projective
cover. This gives rise to the commutative exact diagram
\[\xymatrix@R=10pt@C=15pt{%
   & & 0\ar[d] & 0\ar[d] & & \\
& 0\ar[r] & \Omega(F)\ar[r]\ar[d] & \Omega(F/\rad_\C F)\ar[r]\ar[d] &
\rad_\C F\ar[r] & 0\\
& & \Hom_\C(-,C)\ar@{=}[r]\ar[d] & \Hom_\C(-,C)\ar[d] & & \\
0\ar[r] & \rad_\C F\ar[r] & F\ar[r]\ar[d] & F/\rad_\C F\ar[r]\ar[d] & 0 & \\
   & & 0 & 0 & & }\]
where we have
\begin{align}
\rad_\C^i\Omega(F) & = \rad^{i+1}_\C(-,C)\cap \Omega(F)\notag\\
& = \rad^{i+1}_\C(-,C)\cap\Omega(F/\rad_\C F)\cap \Omega(F)\notag\\
& = \rad_\C^i\Omega(F/\rad_\C F)\cap \Omega(F)\notag
\end{align}
for all $i\geq 1$. It follows that there exist exact sequences
\[0\to \Omega^i(F)\to \Omega^i(F/\rad_\C F)\to \Omega^{i-1}(\rad_\C F)\to
0\] such that $\rad_\C^j\Omega^i(F/\rad_\C F)\cap
\Omega^i(F)=\rad_\C^j\Omega^i(F)$ for all $j\geq 0$. This in turn gives
the exact sequences
\[0\to \Hom(\Omega^{i-1}(\rad_\C F),S_D)\to \Hom(\Omega^i(F/\rad_\C
F),S_D) \to \Hom(\Omega^i(F),S_D)\to 0\]
for each simple functor $S_D$ in $\Mod(\C)$. Since
$\Hom_{\Mod(\C)}(\Omega^j(G),S_D)\simeq \Ext^j_{\Mod(\C)}(G,S_D)$, these
sequences induce the exact sequence
\[0\to \Ext^*_{\Mod(\C)}(\rad_\C F,-)[-1] \to \Ext^*_{\Mod(\C)}(F/\rad_\C
F,-)\to \Ext^*_{\Mod(\C)}(F,-)\to 0\]
of functors. By Corollary \ref{cor:radweaklykoszul} $\rad_\C F$ is weakly
Koszul, so that by induction there exists a long exact sequence
\begin{multline}
\cdots\to \Ext^*_{\Mod(\C)}(\rad_\C^2 F/\rad_\C^3 F,-)[-2] \to
\Ext^*_{\Mod(\C)}(\rad_\C F/\rad_\C^2 F,-)[-1]\to\\
\Ext^*_{\Mod(\C)}(F/\rad_\C F,-)\to \Ext^*_{\Mod(\C)}(F,-)\to 0\notag
\end{multline}
of functors. Note that by Corollary \ref{cor:radweaklykoszul} the
projectives occurring in this resolution are finitely generated.  In
particular $\phi(F)=\Ext^*_{\Mod(\C)}(F,-)$ is a linear module.
\end{proof}

The next result shows that for a weakly Koszul $K$-category $\C$ the
two naturally associated categories $\A_\gr(\C)$ and $E(\S(\C))^\op$
are Koszul duals of each other.
\begin{prop}
Let $\C$ be a weakly Koszul $K$-category, and let $G\colon \Mod(\C)\to
\Gr(\A_\gr(\C))$ be given as before by $G(F)=\amalg_{i\geq 0} \rad_\C^i
F/\rad_\C^{i+1}F$.
\begin{enumerate}
\item[(a)] For $F$ in $\wK(\C))$, then
\[\Ext^*_{\Mod(\C)}(F,-)\simeq
\Ext^*_{\Mod\A_\gr(\C)}(G(F),-)\comp G\]
as objects in $\Gr(E(\S(\C))^\op)$.
\item[(b)] The functor $G$ induces an equivalence of the categories
$E(\S(\C))$ and $E(\S(\A_\gr(\C)))$.
\end{enumerate}
\end{prop}
\begin{proof}
(a) Let $F$ be in $\wK(\C)$, and let
\[\cdots\to\Hom_C(-,C_i)\to \Hom_\C(-,C_{i-1})\to \cdots 
\to \Hom_\C(-,C_0) \to F\to 0\]
be a minimal projective resolution of $F$. By Proposition
\ref{prop:weaklytoKoszul}
\begin{multline}
\cdots\to G(\Hom_C(-,C_i))\to G(\Hom_\C(-,C_{i-1}))\to \cdots\notag\\ 
\to G(\Hom_\C(-,C_0)) \to G(F)\to 0\notag
\end{multline}
is a minimal projective resolution of $G(F)$. It follows from this
that $G(\Omega^i(F))\simeq \Omega^i(G(F))$.

For any pair of objects $C$ and $X$ in $\C$ with $C$ indecomposable we
have that $\Hom_{\Mod(\C)}(\Hom_\C(-,X),S_C) \simeq
\Hom_{\Mod\A_\gr(\C)}(G(\Hom_\C(-,X)),G(S_C))$, where the isomorphism
is induced by $G$. It follows directly from this that for
$F$ in $\wK(\C)$ we have an isomorphism of vector spaces
\[\Ext^i_{\Mod(\C)}(F,S_C)\simeq \Ext^i_{\Mod\A_\gr(\C)}(G(F),G(S_C))\]
for any indecomposable object $C$ in $\C$, where the isomorphism is
induced by $G$.

Finally we need to see that the vector space isomorphism is a morphism
of functors. Let $\theta\colon S_C\to S_D$ be a homogeneous morphism
in $E(\S(\C))$, that is, let $\theta$ be some element in
$\Ext^i_{\Mod(\C)}(S_C,S_D)$ for some $i\geq 0$. Then
\[\Ext^*_{\Mod(\C)}(F,-)(\theta) \colon \Ext^*_{\Mod(\C)}(F,S_C)\to
\Ext^*_{\Mod(\C)}(F,S_D)\]
is given by the Yoneda product with $\theta$. And
\begin{multline}
\Ext^*_{\Mod\A_\gr(\C)}(G(F),-)\comp G(\theta)\colon
\Ext^*_{\Mod\A_\gr(\C)}(G(F),G(S_C))\to \notag\\
\Ext^*_{\Mod\A_\gr(\C)}(G(F),G(S_D))\notag
\end{multline}
is given by the Yoneda product with $G(\theta)$ in
$\Ext^i_{\Mod\A_\gr(\C)}(G(S_C),G(S_D))$. It is easy to see that
vector space isomorphism commutes with these operations, so that
$\Ext^*_{\Mod(\C)}(F,-)$ and
$\Ext^*_{\Mod\A_\gr(\C)}(G(F),-)\comp G$ are isomorphic as objects
in $\Gr(E(\S(\C))^\op)$.

(b) Since $\C$ is weakly Koszul, all simple functors are in
$\wK(\C)$. It then follows from (a) that $E(\S(\C))$ and
$E(\S(\A_\gr(\C)))$ are equivalent categories, where the equivalence
is induced by $G$.
\end{proof}

The above results can be summarized as having a commutative diagram
\[\xymatrix{%
\wK(\C) \ar[rr]^G \ar[dr]^\phi & &
\Li(\A_\gr(\C))\ar[dl]^{\phi'} \\
& \Li(E(\S(\C))^\op) & }\]
where $\Li(\D)$ denotes the full subcategory consisting of the linear
objects for a Koszul category $\D$, the functor
$\phi(F)=\Ext^*_{\Mod(\C)}(F,-)$ and the functor
$\phi'(F')=\Ext^*_{\Mod\A_\gr(\C)}(F',-)$. When $\C$ is weakly
Koszul, then $\A_\gr(\C)$ is Koszul by Proposition
\ref{prop:weaklytoKoszul} and $E(\S(\C))$ is Koszul by Theorem
\ref{thm:koszul}. Furthermore, the functor $\phi'$ is a duality and
$\Gr(E(\S(E(\S(\C))^\op))^\op)$ is equivalent to $\Gr(\A_\gr(\C))$. In other
words, we have the following. 

\begin{prop}
  Let $\C$ be a weakly Koszul $K$-category. Then the double Koszul
  dual $E(\S(E(\S(\C))^\op))^\op$ of the weakly Koszul $K$-category
  $\C$ is equivalent to the associated graded $K$-category
  $\A_\gr(\C)$.
\end{prop}

We end this section with noting that when $\C$ is Koszul, then the
associated graded category is equivalent to $\C$.
\begin{prop}\label{prop:Koszul=assgraded}
Let $\C$ be a Koszul $K$-category. Then $\C$ and $\A_\gr(\C)$ are
equivalent graded $K$-categories.
\end{prop}
\begin{proof}
The objects in $\C$ and $\A_\gr(\C)$ are the same. The homomorphisms
in $\C$ are graded vector spaces $\Hom_\C(C,D)=\amalg_{i\geq
0}\Hom_\C(C,D)_i$. By Lemma \ref{lem:genin0and1} we have
$\rad^i_\C(C,D)=\amalg_{j\geq i}\Hom_\C(C,D)_j$ for any $i\geq
0$. In particular the natural morphism
\[\Hom_\C(C,D)_i\to \rad^i_\C(C,D)/\rad^{i+1}_\C(C,D)\]
is an isomorphism for all objects $C$ and $D$ in $\C$. This is easily
seen to extend to an isomorphism
\[\Hom_\C(C,D)\simeq \amalg_{i\geq
  0}\rad^i_\C(C,D)/\rad^{i+1}_\C(C,D)\]
for all objects $C$ and $D$ in $\C$. Hence it follows that $\C$ and
$\A_\gr(\C)$ are equivalent graded $K$-categories.
\end{proof}

\section{Free tensor categories over a bimodule and Koszul duality}

Let $\C$ be an additive Krull-Schmidt $K$-category, where
$\rad_\C(-,C)$ is a finitely generated functor in $\Mod(\C)$ for all
$C$ in $\C$ and $\A_\gr(\C)$ is locally finite. In this section we
define a free tensor category $T(\C)$ associated to $\C$ over a
bimodule, which is such that if $\C$ is weakly Koszul or Koszul, then
$\A_\gr(\C)$ is a quotient of $T(\C)$ by an ideal $\I$ generated in
degree $2$. When $\C$ is weakly Koszul or Koszul, we show that the
Koszul dual of $\C$ is given by $T(\E(\C))$ modulo the orthogonal
relations of $\I_2$.

Let $\C$ be an additive Krull-Schmidt $K$-category, where
$\rad_\C(-,C)$ is a finitely generated functor in $\Mod(\C)$ for all
$C$ in $\C$ and $\A_\gr(\C)$ is locally finite. Then we define the
free tensor category $T(\C)$ as follows. The objects in $T(\C)$ are
the same as the objects in $\A_\gr(\C)$, and the morphisms in $T(\C)$
are given as
\[\Hom_{T(\C)}(X,Y)=\amalg_{i\geq 0} \Hom_{T(\C)}(X,Y)_i,\]
where
\[\Hom_{T(\C)}(X,Y)_i=
\begin{cases}
\Hom_\C(X,Y)/\rad_\C(X,Y),  & \text{if $i=0$,}\\
\begin{array}{c}
\amalg_{Y_1,\ldots,Y_{i-1}\in \C} \rad_\C(Y_{i-1},Y)/\rad_\C^2(Y_{i-1},Y)
\otimes_{D_{Y_{i-1}}}\\
\cdots\otimes_{D_{Y_2}}\rad_\C(Y_1,Y_2)/\rad_\C^2(Y_1,Y_2)\\ 
\otimes_{D_{Y_{1}}} \rad_\C(X,Y_1)/\rad_\C^2(X,Y_1), 
\end{array}  & \text{if $i\geq 1$,}
\end{cases}\]
viewed inside
\begin{multline}
\amalg_{Y_0,Y_1,\ldots,Y_{i-1},Y_i\in \C}
\rad_\C(Y_{i-1},Y_i)/\rad_\C^2(Y_{i-1},Y_i)\otimes_{D_{Y_{i-1}}}\notag\\
\cdots\otimes_{D_{Y_2}}\rad_\C(Y_1,Y_2)/\rad_\C^2(Y_1,Y_2)\notag\\
\otimes_{D_{Y_{1}}} \rad_\C(Y_0,Y_1)/\rad_\C^2(Y_0,Y_1),\notag
\end{multline}
where $D_Z=\Hom_\C(Z,Z)/\rad_\C(Z,Z)$ for $Z$ in $\C$. The
composition in $T(\C)$ is given by
\[(f_m\otimes\cdots \otimes f_2\otimes f_1)\comp (g_n\otimes\cdots 
\otimes g_2\otimes g_1) =
f_m\otimes\cdots\otimes f_2\otimes f_1\otimes g_n\otimes\cdots 
\otimes g_2\otimes g_1.\]
With these definitions $T(\C)$ is a locally finite graded $K$-category
with radical given by
\[\rad_{T(\C)}(X,Y)=\Hom_{T(\C)}(X,Y)_{\geq 1}.\]
Furthermore, we have a full and dense functor $\phi\colon T(\C)\to
\A_\gr(\C)$ given by
\[\phi(X)=X\]
for $X$ in $T(\C)$ and
\[\phi(f_m\otimes\cdots\otimes f_2\otimes
f_1)=\overline{f_mf_{m-1}\cdots f_2f_1}\]
in $\rad_\C^m(X,Y)/\rad_\C^{m+1}(X,Y)$. This functor is full and dense,
since $\A_\gr(\C)$ is a graded category generated in the degrees $0$
and $1$. The kernel of the functor $\phi$ is an ideal $\I$ in $T(\C)$
satisfying $\I\subseteq \rad^2_{T(\C)}$. Hence, the categories
$T(\C)/\I$ and $\A_\gr(\C)$ are equivalent. 

The next basic property of the construction of a free tensor category
over a bimodule is the following. As for tensor algebras over a
semisimple ring, the free tensor category $T(\C)$ is hereditary, as we
show next.
\begin{lem}
The category $\Gr(T(\C))$ is hereditary.
\end{lem}
\begin{proof}
The simple functors in $\Gr(T(\C))$ are given as
$S_C=\Hom_{T(\C)}(-,C)/\rad_{T(\C)}(-,C)$ for an indecomposable object
$C$ in $T(\C)$. Hence we have the exact sequence
\[0\to \rad_{T(\C)}(-,C)\to \Hom_{T(\C)}(-,C)\to S_C\to 0.\]
We have that $\rad_{T(\C)}(-,C)$ is given by
\begin{multline}
\amalg_{n\geq 1}\amalg_{Y_1,\ldots,Y_{n-1},Y_n\in \C}
\rad_\C(Y_n,C)/\rad_\C^2(Y_n,C) \otimes_{D_{Y_n}}\notag\\
\cdots \otimes_{D_{Y_2}}\rad_\C(Y_1,Y_2)/\rad_\C^2(Y_1,Y_2)\notag\\
\otimes_{D_{Y_1}} 
\rad_\C(-,Y_1)/\rad_\C^2(-,Y_1),\notag
\end{multline}
The functor $\rad_{T(\C)}(-,C)/\rad^2_{T(\C)}(-,C)$ in $\Gr(T(\C))$ is
semisimple, so it is isomorphic to $\amalg_{i=1}^t S_{Z_i}$ for some
indecomposable objects $Z_i$ in $T(\C)$ and
\[\amalg_{Y_n\in T(\C)} \rad_{T(\C)}(Y_n,C)/\rad^2_{T(\C)}(Y_n,C)\simeq
\amalg_{i=1}^t \rad_{T(\C)}(Z_i,C)/\rad^2_{T(\C)}(Z_i,C).\]
Hence we infer that
\[\rad_{T(\C)}(-,C)=\amalg_{i=1}^t
\rad_{T(\C)}(Z_i,C)/\rad^2_{T(\C)}(Z_i,C)\otimes_{D_{Z_i}}\Hom_{\T(C)}(-,Z_i)\]
and conclude that $\rad_{T(\C)}(-,C)$ is a projective object in
$\Gr(T(\C))$. It follows from Theorem \ref{thm:globaldim} that
the category $\Gr(T(\C))$ is hereditary.
\end{proof}

As a consequence of the above considerations we obtain a
generalization of the classical result for Koszul algebras that they
are quadratic (see \cite[Corollary 2.3.3]{BGS} or \cite[Corollary
7.3]{GM2}).

\begin{prop}\label{prop:quadratic}
Let $\C$ be a Koszul $K$-category. Then the category $\C$ is
quadratic, that is, there exists an ideal $\I$ in $T(\C)$ generated in
degree $2$ such that $\C$ and $T(\C)/\I$ are equivalent categories.  
\end{prop}
\begin{proof}
When $\C$ is a Koszul $K$-category, $\C$ and $\A_\gr(\C)$ are
equivalent graded $K$-categories by Proposition
\ref{prop:Koszul=assgraded}. Suppose $\A_\gr(\C)$ is equivalent to
$T(\C)/\I$. Then, for any simple functor $S_C$ in $\Gr(\A_\gr(\C))$,
we have the \emph{Butler resolution} of $S_C$ 
\begin{multline}
0\to \I(-,C)/\I\rad_{T(\C)}(-,C)\to
\rad_{T(\C)}(-,C)/\I\rad_{T(\C)}(-,C)\notag\\
\to \Hom_{\A_\gr(\C)}(-,C)\to S_C\to 0,\notag
\end{multline}
which is a start of a minimal projective resolution of $S_C$ in
$\Gr(\A_\gr(\C))$. Since $\Gr(T(\C))$ is hereditary, the functor
$\rad_{T(\C)}(-,C)$ is a projective functor and we have seen that it
is isomorphic to 
\[\amalg_{i=1}^t\rad_{T(\C)}(Z_i,C)/\rad^2_{T(\C)}(Z_i,C)
\otimes_{D_{Z_i}} \Hom_{T(\C)}(-,Z_i)
\] 
for some indecomposable objects $Z_i$ in $T(\C)$. This implies that we
have the following 
\begin{align}
\rad_{T(\C)}(-,C)/\I\rad_{T(\C)}(-,C) & \simeq
{\Hom_{T(\C)}(-,-)}/I(-,-)\otimes_{T(\C)}\rad_{T(\C)}(-,C)\notag\\
& \simeq {\Hom_{\A_\gr(\C)}(-,-)}\otimes_{T(\C)} \rad_{T(\C)}(-,C)\notag\\
& \simeq \amalg_{i=1}^t{R_{T(\C)}^1(Z_i,C)\otimes_{D_{Z_i}}
  \Hom_{\A_\gr(\C)}(-,Z_i)},\notag 
\end{align}
where $R^1_{T(\C)}(Z_i,C)=\rad_{T(\C)}(Z_i,C)/\rad_{T(\C)}^2(Z_i,C)$.  
Moreover we have that $\Omega^2(S_C)=\I(-,C)/\I\rad_{T(\C)}(-,C)$ and
\[\rad_\C\Omega^2(S_C)=(\rad_{T(\C)}\I(-,C)+
\I\rad_{T(\C)}(-,C))/\I\rad_{T(\C)}(-,C).\]
From this we obtain that
\begin{align}
\Omega^2(S_C)/\rad_\C\Omega^2(S_C) & \simeq
\I(-,C)/(\rad_{T(\C)}\I(-,C)+\I\rad_{T(\C)}(-,C))\notag\\
& = \I_2(-,C)\notag
\end{align}
Since $\A_\gr(\C)$ is a Koszul category,
$\Omega^2(S_C)$ is generated in degree $2$, which in turn implies that
the ideal $\I$ is generated by $\I_2(-,-)$ as a two-sided ideal in
$T(\C)$.
\end{proof} 

Given a Krull-Schmidt category $\C$, form the $\Ext$-category
$\E(\C)$ as in Example \ref{exam:extcat} in Section
\ref{section1}. Then $\E(\C)$ is a Krull-Schmidt category again, and
we can form the free tensor category $T(\E(\C))$ of $\E(\C)$ as
above. The category $\E(\C)$ has as indecomposable objects the simple
functors $S_C$ for $C$ indecomposable in $\C$, and morphisms are given
by
\[\Hom_{\E(\C)}(S_X,S_Y)=\amalg_{i\geq 0} \Ext^i_{\Mod(\C)}(S_X,S_Y).\]
The radical in $\E(\C)$ is given by
$\rad_{\E(\C)}(S_X,S_Y)=\amalg_{i\geq 1} \Ext^i_{\Mod(\C)}(S_X,S_Y)$, so
that
\[\rad_{\E(\C)}(S_X,S_Y)/\rad^2_{\E(\C)}(S_X,S_Y)=\Ext^1_{\Mod(\C)}(S_X,S_Y).\]
Since $\E(\C)$ is a graded category with
$\Hom_{\E(\C)}(S_X,S_Y)_i=\rad^i_{\E(\C)}(S_X,S_Y)/\rad^{i+1}_{\E(\C)}(S_X,S_Y)$,
the associated graded category of $\E(\C)$ is the same as
$\E(\C)$. Therefore, as pointed out above, there is a full and dense
functor $\phi\colon T(\E(\C))\to \E(\C)$ with a kernel $\I'$ contained
in $\rad^2_{T(\E(\C))}$. Furthermore, $T(\E(\C))/\I'$ and $\E(\C)$ are
equivalent graded categories. Our next aim is to show that the ideal
$\I'$ of relations in $T(\E(\C))$ is obtained as "the orthogonal
relations" of the relations for the presentation of $\A_\gr(\C)$ as a
quotient $T(\C)/\I$, whenever $\A_\gr(\C)$ is a Koszul category. See
\cite[Corollary 2.3.3]{BGS} or \cite[Corollary 7.3]{GM1} for
corresponding result for algebras. 

To justify the claim in the proposition below observe that for $C$ and 
$Z$ indecomposable in $\C$ the functor
$U=D(\rad_{\A_\gr(\C)}(-,C)/\rad^2_{\A_\gr(\C)}(-,C))$ is isomorphic to
\begin{align}
U 
& \simeq D(\Hom_\C(-,-)/\rad_\C(-,-)\otimes_\C\rad_{\A_\gr(\C)}(-,C)) \notag\\
& \simeq \Hom(\rad_{\A_\gr(\C)}(-,C),D(\Hom_\C(-,-)/\rad_\C(-,-)))\notag\\
& \simeq \Hom(\rad_{\A_\gr(\C)}(-,C),\Hom_\C(-,-)/\rad_\C(-,-))\notag\\
& \simeq
\Hom(\rad_{\A_\gr(\C)}(-,C)/\rad^2_{\A_\gr(\C)}(-,C),\Hom_\C(-,-)/\rad_\C(-,-)), 
\notag  
\end{align}
and evaluating this isomorphism at $Z$ we get 
\[D(\rad_{\A_\gr(\C)}(Z,C)/\rad^2_{\A_\gr(\C)}(Z,C))   
\simeq \Ext^1_{\Mod(\C)}(S_C,S_Z).\]
From this we infer that 
\[D(\rad_{\A_\gr(\C)}(-,Z)/\rad^2_{\A_\gr(\C)}(-,Z)) \otimes_{D_{Z}} 
D(\rad_{T(\C)}(Z,C)/\rad^2_{T(\C)}(Z,C)) 
\] 
can be viewed as 
\[\Ext^1_{\Mod(\C)}(S_{Z},S_{(-)})\otimes_{D_{Z}} \Ext^1_{\Mod(\C)}(S_C,S_{Z})\]
that is, contained in $\rad_{T(E(\C))}^2(S_C,-)$. 

\begin{prop}
Let $\C$ be a Krull-Schmidt category, and assume that $\A_\gr(\C)$ is
Koszul. If $\A_\gr(\C)$ is equivalent to $T(\C)/\I$, then $\E(\C)$ is
equivalent to $T(\E(\C))/\I'$, where $\I'=\langle \I_2^\perp\rangle$
with $\I_2^\perp(-,C)$ being given as the kernel of the natural morphism
\[
\begin{array}{c}
\amalg_{i=1}^t D(\rad_{\A_\gr(\C)}(-,Z_i)/\rad^2_{\A_\gr(\C)}(-,Z_i))%
\otimes_{D_{Z_i}}D(\rad_{T(\C)}(Z_i,C)/\rad^2_{T(\C)}(Z_i,C)) \\
\downarrow \\
 D(\I_2(-,C)),
\end{array}
\]
for each indecomposable object $C$ in $\C$ and some indecomposable
objects $Z_i$ in $\C$ for $i=1,2,\ldots,t$. Here
$D_{Z}=\End_\C(Z)/\rad_\C(Z,Z)$.
\end{prop}
\begin{proof}
There are long formulas in the proof of this result. So to save space
we use the following short versions of the following spaces when
convenient. Let $R_{?}^i(-,-)=\rad^i_{?}(-,-)/\rad^{i+1}_{?}(-,-)$,
where $?$ denotes the category where the radical is given. Morphism
spaces $\Hom_{?}(-,-)$ are shorten to $_{?}(-,-)$. 

Suppose $\A_\gr(\C)$ is equivalent to $T(\C)/\I$.  We want to
describe the relations $\I'$ in $\T(\E(\C))$  such that $\E(\C)$ is
equivalent to $\T(\E(\C))/\I'$. To this end consider the following
commutative diagram. The first column is obtained for an
indecomposable object $C$ in $\C$ as in the proof of Proposition
\ref{prop:quadratic}. 
\[\xymatrix@R=20pt@C=15pt{
0\ar[d] & 0\ar[d]\\
\I(-,C)/\I\rad_{T(\C)}(-,C)\ar[d]\ar[r] &
\amalg_{i=1}^tR_{T(\C)}^1(Z_i,C)\otimes_{D_{Z_i}}\rad_{\A_\gr(\C)}(-,Z_i) \ar[d]\\
\amalg_{i=1}^tR^1_{T(\C)}(Z_i,C)\otimes_{D_{Z_i}} {_{\A_\gr(\C)}(-,Z_i)}
\ar[d]\ar@{=}[r] & 
\amalg_{i=1}^tR_{T(\C)}^1(Z_i,C)\otimes_{D_{Z_i}} {_{\A_\gr(\C)}(-,Z_i)} \ar[d] \\
\rad_{\A_\gr(\C)}(-,C)\ar[d]\ar[r] & R_{\A_\gr(\C)}^1(-,C)\ar[d]\\
0 & 0
}\]
where the middle horizontal morphism is the identity, the lower
horizontal morphism is the natural projection, and the upper
horizontal morphism is the induced inclusion. 
The Snake Lemma then gives rise to the exact sequence
\begin{multline}
0\to \I(-,C)/\I\rad_{T(\C)}(-,C)\to\notag\\
\amalg_{i=1}^tR_{T(\C)}^1(Z_i,C)\otimes_{D_{Z_i}} \rad_{\A_\gr(\C)}(-,Z_i)
\to\notag\\
\rad^2_{\A_\gr(\C)}(-,C)\to 0\notag
\end{multline}
Since $\A_\gr(\C)$ is Koszul, we have that
\[\rad^2(\rad_{T(\C)}(-,C)/\I\rad_{T(\C)}(-,C))\cap \Omega^2(S_C)=
\rad\Omega^2(S_C),\]
hence
\begin{multline}
\left(\amalg_{i=1}^tR_{T(\C)}^1(Z_i,C)\otimes_{D_{Z_i}}\rad^2_{\A_\gr(\C)}(-,Z_i)
\right)\cap \I(-,C)/\I\rad_{T(\C)}(-,C)\notag\\
\simeq (\rad_{T(\C)}\I(-,C)+\I\rad_{T(\C)}(-,C))/\I\rad_{T(\C)}(-,C)\notag
\end{multline}
From this we obtain the exact sequence
\[0\to \I_2(-,C)\to
\amalg_{i=1}^t R^1_{T(\C)}(Z_i,C)\otimes_{D_{Z_i}}R_{\A_\gr(\C)}^1(-,Z_i)
\to R_{\A_\gr(\C)}^2(-,C)\to 0\]
This is a sequence of semisimple functors, so that applying
$\Hom(-,S_X)$ we get the following exact sequence
\begin{multline}
0\to \Hom(R_{\A_\gr(\C)}^2(-,C),,S_X)\to \notag\\
\Hom(\left(\amalg_{i=1}^tR^1_{T(\C)}(Z_i,C)\otimes_{D_{Z_i}} 
R_{\A_\gr(\C)}^1(-,Z_i)\right),S_X)\to\notag\\
\Hom(\I_2(-,C),S_X)\to 0\notag
\end{multline}
Using Proposition \ref{prop:duality} (c), the fact that $D(S_X)\simeq
S_X^\op$ and Lemma
\ref{lem:tensormoduloideal}, this sequence is isomorphic to the
sequence
\begin{multline}
0\to D(R_{\A_\gr(\C)}^2(X,C))
\to\tag{$\ddagger$}\label{seq:orthrel}\\
\amalg_{i=1}^tD(R^1_{T(\C)}(Z_i,C) \otimes_{D_{Z_i}} R_{\A_\gr(\C)}^1(X,Z_i))
\to \notag\\
D(\I_2(X,C))\to 0.\notag
\end{multline}
The middle term of this sequence is isomorphic to 
\[D(R_{\A_\gr(\C)}^1(X,Z_i))\otimes_{D_{Z_i}}D(R^1_{T(\C)}(Z_i,C)).\]
As we observed before this proposition, we have
\[D(R_{\A_\gr(\C)}^1(Z_i,C)) \simeq \Ext^1_{\Mod(\C)}(S_C,S_{Z_i})\]
for all $i=1,2,\ldots,t$ and furthermore
\begin{align}
\Ext^2_{\Mod(\C)}(S_C,S_X) & \simeq \Hom(\I(-,C)/\I\rad_{T(\C)}(-,C),S_X)\notag\\
& \simeq \Hom(\I_2(-,C),S_X)\notag\\
& \simeq D(\I_2(X,C)).\notag
\end{align}
With these identifications we can rewrite the sequence \eqref{seq:orthrel} as
\begin{multline}
0\to D(R_{\A_\gr(\C)}^2(X,C))\to\notag\\
\amalg_{i=1}^t\Ext^1_{\Mod(\C)}(S_{Z_i},S_X)\otimes_{D_{Z_i}}
\Ext^1_{\Mod(\C)}(S_C,S_{Z_i})\to\notag\\
\Ext^2_{\Mod(\C)}(S_C,S_X)\to 0,\notag
\end{multline}
and the non-zero epimorphism in the exact sequence \eqref{seq:orthrel}
corresponds to the Yoneda product in the latter exact
sequence. This shows that
$\I^\perp_2=\I'_2$, where $\I'$ is the kernel of the functor
$T(\E(\C))\to \E(\C)$. Since $\E(\C)$ is Koszul, the ideal $\I'$ is
generated in degree $2$, we have $\I'=\langle \I'_2\rangle = 
\langle \I^\perp_2\rangle$.
\end{proof}


\begin{thebibliography}{GM2}
\bibitem[AF]{AF} Anderson, F., Fuller, K., \emph{Rings and categories
  of modules}, Graduate Texts in Mathematics, Vol.\
  13. Springer-Verlag, New-York-Heidelberg, 1992.

\bibitem[A]{A} Auslander, M., \emph{Representation theory of Artin
  algebras \emph{I}}, Comm.\ Algebra 1 (1974), 177--268.

\bibitem[AR1]{AR1} Auslander, M., Reiten, I.,
\emph{Representation theory of artin algebras \emph{III}. Almost split
sequences}, Communications in Algebra \textbf{3} (1975), 239--294.

\bibitem[ARS]{ARS} Auslander, M., Reiten, I., Smal\o, S.\ 
O.; \emph{Representation Theory of Artin Algebras}, Cambridge Studies
in Advanced Mathematics, \textbf{36}, Cambridge University Press,
Cambridge, (1995).

\bibitem[BCS]{BCS} Bautista, R., Colavita, L., Salmer\'on, L.,
  \emph{On adjoint functors in representation theory}, Representations
  of algebras (Puebla, 1980), pp.\ 9--25, Lecture Notes in Math.\ 903.

\bibitem[BGS]{BGS} Beilinson, A., Ginzburg, V., Soergel,
W., \emph{Koszul duality patterns in representation theory},
Journal of the American Mathematical Society, \textbf{9}, no.\ 2.,
(1996), 473--527.

\bibitem[B]{B} Berger, R., \emph{Koszulity for non-quadratic algebras},
  J.\ Algebra 239 (2001), no.\ 2, 705--734. 

\bibitem[CS]{CS} Cassidy, T., Shelton, B., \emph{Generalizing the
  notion of Koszul algebra}, Math.\ Z.\ 260 (2008), no.\ 1, 93--114. 

\bibitem[GM1]{GM1} Green, E.\ L., Mart\'inez-Villa, R., 
\emph{Koszul and Yoneda algebras}, Representation theory of algebras
(Cocoyoc, 1994), 247--297, CMS Conf.\ Proc., \textbf{18},
Amer.\ Math.\ Soc., Providence, RI, (1996).

\bibitem[GM2]{GM2} Green, E.\ L., Mart\'inez-Villa, R., 
\emph{Koszul and Yoneda algebras II}, Algebras and modules, II
(Geiranger, 1996), 227--244, CMS Conf.\ Proc., \textbf{24},
Amer.\ Math.\ Soc., Providence, RI, (1998).

\bibitem[IT]{IT} Igusa, K., Todorov, G., \emph{Radical layers of
  representable functors}, J.\ Algebra 89 (1984), no.\ 1, 105--147.

\bibitem[MVZ]{MVZ} Mart\'inez-Villa, R., Zacharia, D., \emph{Approximations
with modules having linear resolutions}, J.\ Algebra 266 (2003), no.\
2, 671--697.

\bibitem[MOS]{MOS} Mazorchuk, V., Ovsienko, S., Stroppel, C.,
  \emph{Quadratic duals, Koszul dual functors and applications}, to
  appear in Trans.\ Amer.\ Math.\ Soc.

\bibitem[M]{M} Mitchel, B., \emph{Rings with several objects},
  Advances in Math.\ 8 (1972), 1--161.

\bibitem[NvO]{NvO} N\v{a}st\v{a}sescu, C., Van Oystaeyen, F.,
  \emph{Graded and Filtered Rings and Modules}, Lecture Notes in
  Mathematics, 758. Springer, Berlin, 1979. 

\bibitem[P]{P} Priddy, S.\ B., \emph{Koszul resolutions},
  Trans.\ Amer.\ Math.\ Soc.\ 152 (1970) 39--60. 
\end{thebibliography}
\end{document}